\documentclass[11pt]{article}

\usepackage{url}
\usepackage{amssymb,enumerate}
\usepackage{amsmath,amsfonts}
\usepackage{amsthm}
\usepackage{latexsym}
\usepackage{mathrsfs}
\usepackage{color}
\usepackage{microtype}
\usepackage{geometry}
\usepackage{bm}
\usepackage{pdflscape}
\usepackage{afterpage}

\usepackage{multirow}
\allowdisplaybreaks[4]
\usepackage{xcolor}

\theoremstyle{plain}
\newtheorem{theorem}{Theorem}

\newtheorem{lemma}{Lemma}

\theoremstyle{definition}
\newtheorem{definition}{Definition}

\newtheorem{Assumption}{Assumption}

\newcommand{\Rmn}[1]{\uppercase\expandafter{\romannumeral#1}}

\numberwithin{equation}{section}

\newcommand{\Mcal}{\mathcal{M}}

\newcommand{\iprod}[2]{\left \langle #1, #2 \right \rangle }

\newcommand{\tr}{\mathrm{Tr}}

\usepackage{graphicx, color}

\newcommand{\R}{\mathbb{R}}

\newcommand{\prox}{\mathrm{prox}}
\newcommand{\st}{\mathrm{s.t.}}
\newcommand{\be}{\begin{equation}}
\newcommand{\ee}{\end{equation}}
\newcommand{\bea}{\begin{eqnarray}}
\newcommand{\eea}{\end{eqnarray}}
\usepackage{algorithm, algpseudocode}
\usepackage{listings}       
\usepackage{mathrsfs} 
\newcommand{\argmin}{\mathop{\mathrm{arg\,min}}}

\usepackage{lipsum}
\usepackage{url}
\title{SSNCVX: A primal-dual semismooth Newton method for convex composite optimization problem}

\author{Zhanwang Deng\thanks{Center of Machine learning, Peking University (email: { dzw\_opt2022@stu.pku.edu.cn}).}
\and Tao Wei \thanks{Center of Machine learning, Peking University (email: { weit@pku.edu.cn})}
\and Jirui Ma \thanks{Beijing International Center for Mathematical Research, Peking University (email: { majirui@pku.edu.cn}) }
\and Zaiwen Wen \thanks{Beijing International Center for Mathematical Research, Center for Machine Learning Research,
        Changsha Institute for Computing and Digital Economy,
        Peking University (email: { wenzw@pku.edu.cn}).}
}

\date{
\today
}

\begin{document}
\maketitle
	
\begin{abstract}
 In this paper, we propose a uniform semismooth Newton-based algorithmic framework called SSNCVX for solving a broad class of convex composite optimization problems.
 By exploiting the augmented Lagrangian duality, we reformulate the original problem into a saddle point problem and characterize the optimality conditions via a semismooth system of nonlinear equations. The nonsmooth structure is handled internally without requiring problem specific transformation or introducing auxiliary variables. This design allows easy modifications to the model structure, such as adding linear, quadratic, or shift terms through simple interface-level updates. The proposed method features a single loop structure that simultaneously updates the primal and dual variables via a semismooth Newton step. Extensive numerical experiments on benchmark datasets show that SSNCVX outperforms state-of-the-art solvers in both robustness and efficiency across a wide range of problems.

\end{abstract}

\section{Introduction}

In this paper,  we aim to develop an algorithmic framework for the following convex composite problem:

\begin{equation} \label{general}
\begin{aligned}
\min_{\bm{x} } &\quad p(\bm{x}) + f(\mathcal{B}(\bm{x})) + \iprod{\bm{c}}{\bm{x}} + \frac{1}{2}\iprod{\bm{x}}{\mathcal{Q}(\bm{x})}    , \\
\text{s.t.}  & \quad  \bm{x} \in \mathcal{P}_1,~~ \mathcal{A}(\bm{x}) \in \mathcal{P}_2,
\end{aligned}
\end{equation}
where $p(\bm{x})$ is a convex and nonsmooth function,  $\mathcal{A}: \mathcal{X} \rightarrow \mathbb{R}^m, \mathcal{B}: \mathcal{X} \rightarrow \mathbb{R}^l $ are linear operators, $f: \mathbb{R}^l \rightarrow \mathbb{R}$ is a convex function, $\bm{c} \in \mathcal{X}$, $\mathcal{Q}$ is a positive semidefinite matrix or operator, $\mathcal{P}_1 = \{\bm{x} \in \mathcal{X} | \texttt{l} \le \bm{x} \le \texttt{u} \}$ and $ \mathcal{P}_2 = \{\bm{x}\in \mathbb{R}^m|\texttt{lb} \le \bm{x} \le \texttt{ub}\}$.   The choices of $p(\bm{x})$ provide flexibility to handle many types of problems. While the model \eqref{general} focuses on a single variable $\bm{x}$, it is indeed capable of solving the following more general problem with $N$ blocks of variables with shifting terms $\bm{b}_{1,i}$ and $\bm{b}_{2,i}, (i =1,\cdots,N)$:
\begin{equation} \label{general-block}
\begin{aligned}
\min_{\bm{x}_i } &\quad \sum_{i=1}^N p_i(\bm{x}_i - \bm{b}_{1,i} ) + \sum_{i=1}^N f_i(\mathcal{B}_i(\bm{x}) - \bm{b}_{2,i})  + \sum_{i=1}^N \iprod{\bm{c}_i}{\bm{x}_i} + \sum_{i=1}^N \frac{1}{2}\iprod{\bm{x}_i}{\mathcal{Q}_i(\bm{x}_i)}  , \\
\text{s.t.}  & \quad  \bm{x}_i \in \mathcal{P}_{1,i},~~ \sum_{i=1}^N \mathcal{A}_i (\bm{x}_i) \in \mathcal{P}_{2,i}, \quad i =1,\cdots,N,
\end{aligned}
\end{equation}
where $p_i,f_i,c_i,\mathcal{Q}_i, \mathcal{P}_{1,i}$, and $\mathcal{P}_{2,i}$ satisfy the same assumptions in \eqref{general}.
Models \eqref{general} and \eqref{general-block} have widespread applications in engineering, image processing, and machine learning, etc.  We refer the readers to \cite{boyd2011distributed,anjos2011handbook,wolkowicz2012handbook,dantzig2002linear,ben2001lectures} for more concrete applications.

\subsection{Related works}

The first-order methods are popular for solving \eqref{general} due to the easy implementation and rapid convergence speed to a moderate accuracy point. For SDP and SDP+ problems, the alternating direction method of multipliers (ADMM), as implemented in SDPAD \cite{wen2010alternating}, has demonstrated considerable numerical efficiency. A convergent symmetric Gauss–Seidel based three-block ADMM method is developed in \cite{sun2015convergent}, which is capable of handling SDP problems with additional polyhedral set constraints. ABIP and ABIP+ \cite{deng2022new}  are new interior point methods for conic programming. ABIP
uses a few steps of ADMM to approximately solve the subproblems that arise when applying a path-following barrier algorithm to the homogeneous self-dual embedding of the problem. SCS \cite{o2021operator,o2016conic} is an ADMM-based solver for convex quadratic cone programs implemented in C that applies ADMM to the homogeneous self-dual embedding of the problem, which yields infeasibility
certificates when appropriate. TFOCS \cite{becker2011templates} and FOM \cite{beck2019fom} are solvers that aim to solve convex composite optimization problems using a class of first-order
algorithms such as the Nesterov-type accelerated method.  

The interior point method (IPM) is a classical approach for solving a subclass of \eqref{general}, particularly for conic programming. There are well-designed open source solvers based on the interior point methods, such as {SeDuMi} \cite{sturm1999using} and {SDPT3} \cite{toh1999sdpt3}. For commercial solvers, {MOSEK} \cite{mosek} is a high-performance optimization package specializing in large-scale convex problems (e.g., LP, QP, SOCP, SDP). Another state-of-the-art solver, {Gurobi} \cite{optimization2021gurobi}, excels in speed and scalability for complex optimization tasks, including LP, SOCP, and QP.
Building on these solvers, {CVX} \cite{grant2008cvx} is a MATLAB-based modeling framework for convex optimization, while its Python counterpart {CVXPY} \cite{diamond2016cvxpy} offers similar functionality.
 When addressing conic constraints in \eqref{general}, the interior point methods rely on smooth barrier functions to ensure that the iterates lie within the cone. If direct methods are used to solve the linear equation, each iteration of IPMs requires factorizing the Schur complement matrix, which becomes increasingly costly in both computational and memory as the constraint dimension of the problem grows. Moreover, when iterative methods are used in this context, they often fail to exploit the sparse or low-rank structure of the solution. Furthermore, for general nonsmooth terms, the interior point methods cannot handle them directly.
For instance, problems involving $\|\bm{x}\|_1$ are typically first reformulated as linear programs and then solved using interior-point methods \cite{becker2011templates, candes2007dantzig}.

The semismooth Newton (SSN) methods are also effective for solving certain subclasses of problems in \eqref{general}, such as Lasso \cite{li2018highly,xiao2018regularized} and semidefinite programming (SDP) \cite{li2018semismooth,yang2015sdpnal}. One class of  SSN methods integrates SSN into the augmented Lagrangian method (ALM) framework to solve subproblems of the primal variable, such as SDPNAL+ \cite{sun2020sdpnal+} for SDP with bound constraints and SSNAL~\cite{li2018highly}  for Lasso problems. In addition, SSN can also be applied directly to solve a single nonlinear system derived from optimality conditions. A regularized semismooth Newton method is proposed \cite{xiao2018regularized} to solve two-block composite optimization problems such as Lasso and basis pursuit problems. Based on the equivalence of DRS iteration and ADMM, an efficient solver named SSNSDP for SDP is designed \cite{li2018semismooth}. The idea is further extended to solving optimal transport problems \cite{liu2022multiscale}. However, their analysis of superlinear convergence relies on the BD regularity, which indicates that the solution is isolated. To alleviate this problem, based on the strict complementary and local error bound condition, the superlinear linear convergence of regularized SSN for composite optimization is proposed in \cite{hu2025analysis}. Algorithms based on DRS or proximal gradient mapping can only handle two-block problems.  To alleviate this problem, based on the saddle point problems induced from the augmented Lagrangian duality~\cite{deng2025augmented}, an efficient method called ALPDSN is designed for multi-block problems. It also demonstrates considerable performance on various
SDP benchmarks \cite{deng2025efficient}.
 A decomposition method called SDPDAL \cite{wang2023decomposition} is employed to handle SDP and QSDP with bound constraints, where the subproblem is solved using a semismooth Newton approach.
Compared with the interior point methods, the semismooth Newton methods make use of the intrinsic sparse or low-rank structure efficiently, resulting in low memory requirements and low computational cost at each iteration.  Therefore, developing a convex optimization framework specifically designed for multi-block practical applications is of theoretical and practical significance.


\subsection{Contribution}
We develop an SSN-based general-purpose optimization framework for solving the broad class of problems described in Model \eqref{general}.
The contributions of this paper are listed as follows.
\begin{itemize}
\item  A practical model encompasses various optimization problems with nonsmooth terms or constraints (see Table~\ref{tabel-problem-summarize} for details). By leveraging the AL duality, we transform the original problem \eqref{general} into a saddle point problem and formulate a semismooth system of nonlinear equations to characterize the optimality conditions. Unlike the interior point methods, our framework handles nonsmooth terms such as coupling conic constraints and simple norm constraints in standard form, without additional relaxation variables. Furthermore, it is more user-friendly, allowing for easy modifications to the optimization model, such as adding linear, quadratic, or shift terms. Instead of designing separate algorithms for each problem, the proposed framework requires only the selection of different functions and constraints, with updates made solely at the interface level.

\item A unified algorithmic framework can handle complex multi-block semismooth systems.
Unlike some SSN-based methods that rely on switching to first-order steps (e.g., fixed point iteration or ADMM) to ensure convergence, our approach retains second-order information at every iteration, ensuring faster and more robust convergence. Furthermore,  we introduce a systematic approach for calculating generalized Jacobians, enabling efficient second-order updates for a broad class of nonsmooth functions. For certain complex non-smooth functions, we provide the detailed derivations of computationally efficient implementations. These effective computational approaches enable the practical utilization of both low-rank and sparse structures within the corresponding non-smooth functions.

\item Comprehensive and promising numerical results. To rigorously evaluate the performance of SSNCVX, we conduct extensive experiments across a wide range of optimization problems, including Lasso, fused Lasso, SOCP, QP, and SPCA problems. SSNCVX demonstrates superior performance compared to state-of-the-art solvers on all these problems. These results not only validate SSNCVX as a highly efficient and reliable solver but also underscore its potential as a versatile tool for large-scale optimization tasks in related fields such as machine learning and signal processing.

\end{itemize}

\subsection{Notation}
For a linear operator $\mathcal{A}$, its adjoint operator is denoted by $\mathcal{A}^*$.  For a proper convex function $g$,  we define its domain as ${\rm dom}(g):=\{ \bm{x}: g(\bm{x}) < \infty\}$. The Fenchel conjugate function of $g$ is $g^*(\bm{z}) := \sup_{\bm{x}}\{\left<\bm{x},\bm{z}\right> - g(\bm{x})  \}$ and the subdifferential  is $ \partial g(\bm{x}): = \{\bm{z}:~ g(\bm{y}) - g(\bm{x}) \geq \left<\bm{z}, \bm{y} - \bm{x}  \right>,~\forall \bm{y}  \}. $
For a convex set $\mathcal{Q}$, we use the notation $\delta_{\mathcal{Q}}$ to denote the indicator function of the set $\mathcal{Q}$, which takes the value $0$ on $\mathcal{Q}$ and $+\infty$ elsewhere. The relative interior of $ \mathcal{Q}$ is denoted by ${\rm ri}(\mathcal{Q})$.
 For any proper closed convex function $g$, and constant $t>0$, the proximal operator of $g$ is defined by $
    \prox_{tg}(\bm{x}) = \arg\min_{\bm{y}}\{g(\bm{y}) + \frac{1}{2t}\|  \bm{y} - \bm{x}\|^2  \}.
$
The Moreau envelope function of $g$ is defined as $e_{t} g(x) = \min_{\bm{y}}\{g(\bm{y}) + \frac{t}{2}\|  \bm{y} - \bm{x}\|^2  \}.$
When $g = \delta_{\mathcal{C}}(\bm{x})$ is the indicator function of a convex set $\mathcal{C}$, it holds that ${\rm prox}_{tg}(\bm{x}) = \Pi_{\mathcal{C}}(\bm{x})$, where $\Pi_{\mathcal{C}}$ denotes the projection onto the set $\mathcal{C}$.

\subsection{Organization}
The rest of this paper is organized as follows. A primal-dual semismooth Newton method based on the AL duality is introduced in Section \ref{2}.  The properties of the proximal operator are introduced in Section \ref{3}. Extensive experiments on various problems are conducted in Section \ref{4} and we conclude this paper in Section \ref{5}.
\afterpage{
\begin{landscape}
\begin{table}[ht]
\centering
\footnotesize
\resizebox{1.3\textwidth}{!}{
\begin{tabular}{|p{2cm}|c|c|c|}
\hline
\textbf{Problem} & \textbf{Objective Function} & \textbf{Constraints} & \textbf{Function block} \\
\hline
   & $\underbrace{\iprod{\bm{c}}{\bm{x}}}_{\text{(I)}} + \underbrace{\frac{1}{2}\iprod{\bm{x}}{\mathcal{Q}(\bm{x})}}_{\text{(II)}} + \underbrace{f(\mathcal{B}(\bm{x}))}_{\text{(III)}} + p(\bm{x})$ & $ \underbrace{\bm{x} \in \mathcal{P}_1}_{\text{(IV)}}$, $\underbrace{\mathcal{A}(\bm{x}) \in \mathcal{P}_2}_{\text{(V)}}.$ & (I) (II) (III) (IV) (V)  \\
\hline
 LP & $ \iprod{\bm{c}}{\bm{x}} $ & $\mathcal{A}(\bm{x}) = \bm{b},  \bm{x} \ge 0$ & (I)(V) \\
\hline
 SOCP & $\iprod{\bm{c}}{\bm{x}}$ & $\mathcal{A}(\bm{x}) = \bm{b},  \bm{x} \in \mathcal{Q}^n$ & (I)(V) \\
\hline
 SDP & $\iprod{\bm{C}}{\bm{X}}$ & $\mathcal{A}(\bm{X}) = \bm{b},  \bm{X} \succeq 0$ & (I)(V) \\
\hline
 SDP with box constraints & $ \iprod{\bm{C}}{\bm{X}} $ & $\mathcal{A}(\bm{X}) = \bm{b},  \bm{x} \in \mathcal{P}_1, \bm{X} \succeq 0$ & (I)(IV)(V) \\
\hline
 QP & $ \iprod{\bm{x}}{\mathcal{Q}(\bm{x})} + \iprod{\bm{x}}{\bm{c}}$ & $\texttt{l} \le \bm{x} \le \texttt{u}, \mathcal{A}(\bm{x}) = \bm{b}$ & (I)(II)(IV)(V) \\
\hline
  QP with $\ell_1$ norm & $ \iprod{\bm{x}}{\mathcal{Q}(\bm{x})} + \lambda \|\bm{x}\|_1 $ & $\texttt{l} \le \bm{x} \le \texttt{u}, \mathcal{A}(\bm{x}) = \bm{b}$ & (I)(II)(III)(V) \\
\hline
 Lasso & $\frac{1}{2}\|\mathcal{B}(\bm{x})-\bm{b} \|^2+ \lambda \|\bm{x}\|_1$ & -  & (III) \\
\hline
 Fused Lasso & $\frac{1}{2}\|\mathcal{B}(\bm{x}) -\bm{b} \|^2 + \lambda_1 \|\bm{x} \|_1 + \lambda_2\|D\bm{x}\|_1$ & - & (III) \\
\hline
 Group Lasso & $\frac{1}{2}\|\mathcal{B}(\bm{x})-\bm{b} \|^2+ \lambda \|\bm{x}\|_2$ &-& (III) \\
\hline
 Top-k Lasso & $ \frac{1}{2}\|\mathcal{B}(\bm{x}) -\bm{b} \|^2 +  \lambda \sum_{i=1}^k \bm{x}_{[i]} $ &-& (III) \\
\hline
Low-rank matrix recovery & $ \|\mathcal{B}(\bm{X}) - \bm{B}\|_{\mathrm{F}}^2 + \lambda \|\bm{X}\|_*$ & - & (III) \\
\hline
  Sparse covariance matrix estimation & $  - \log(\text{det}(\bm{X})) + \tr(\bm{XS}) + \lambda \|\bm{X}\|_1 $ & - & (I)(III) \\
\hline
 Sparse PCA & $  - \iprod{\bm{L}}{\bm{x}} + \lambda \|\bm{x} \|_1   $ & $\tr(\bm{x}) =1, \bm{x} \succeq 0$ & (III) \\
\hline
  Basis pursuit & $  \|\bm{x}\|_1   $ & $ \mathcal{A}(\bm{x}) = \bm{b}  $ & (V) \\
\hline
  Robust PCA & $  \|\bm{x}_1\|_* + \lambda \|\bm{x}_2\|_1   $ & $ \bm{x}_1 + \bm{x}_2  = \bm{D} $ & (III)(V) \\
\hline
\end{tabular}
}\caption{Examples that Model \eqref{general} is able to solve.}
\label{tabel-problem-summarize}
\end{table}
\end{landscape}
}

\section{A primal-dual semismooth Newton method} \label{2}
In this section, we introduce a primal-dual semismooth Newton method to solve the original problem \eqref{general}.
 We first transform \eqref{general} into a saddle point problem using the AL duality in Section \ref{2-1}. Subsequently, a monotone nonlinear system induced by the saddle point problem is presented. Such a nonlinear system is semismooth and equivalent to the Karush–Kuhn–Tucker (KKT) optimality condition of problem \eqref{general}. We then introduce an SSN method to solve the nonlinear system in Section \ref{2-2}.
 The efficient calculation of the Jacboian matrix to solve the linear system is introduced in Section \ref{2-3}
 and some implementation details of the algorithm are presented in Section \ref{2-4}.
\subsection{An equivalent saddle point problem} \label{2-1}
The procedure of handing \eqref{general} is similar to that of \cite{deng2025augmented}. However, as the problem being dealt with is more practical and complex, we provide the full algorithmic derivation below for both completeness and reader comprehension.
The dual problem of \eqref{general} can be represented by
\begin{equation} \label{general-dual}
\begin{aligned}
    &\min_{\bm{y},\bm{z},\bm{s},\bm{r},\bm{v}}  \quad  \delta_{\mathcal{P}_2}^*(-\bm{y}) + f^*(\bm{-z}) +  p^*(-\bm{s}) + \frac{1}{2} \iprod{\mathcal{Q}\bm{v}}{\bm{v}} + \delta_{\mathcal{P}_1}^*(-\bm{r}), \\
    &\quad\st \quad  \mathcal{A}^*(\bm{y}) + \mathcal{B}^*\bm{z} + \bm{s} - \mathcal{Q}\bm{v} + \bm{r} = \bm{c}.
    \end{aligned}
\end{equation}
Introducing the slack variables $\bm{o},\bm{q},\bm{t}$, the equivalent optimization problem is
\begin{equation} \label{general-dual2}
\begin{aligned}
&\min_{\bm{y},\bm{z},\bm{s},\bm{r},\bm{v},\bm{o},\bm{q},\bm{t}} \quad  \delta_{\mathcal{P}_2}^*(-\bm{o}) + f^*(\bm{-q}) - \iprod{\bm{b}_1}{\bm{s}} +  p^*(-\bm{s}) + \frac{1}{2} \iprod{\mathcal{Q}\bm{v}}{\bm{v}} + \delta_{\mathcal{P}_1}^*(-\bm{t}), \\
    &\qquad \st \quad  \mathcal{A}^*(\bm{y}) + \mathcal{B}^*\bm{z} + \bm{s} - \mathcal{Q}\bm{v} + \bm{r} = \bm{c},~~\bm{y}=\bm{o},~~\bm{z}=\bm{q},~~\bm{r}=\bm{t}.
    \end{aligned}
\end{equation}
The augmented Lagrangian function of \eqref{general-dual2} is
\begin{equation*}
\begin{aligned}
& \mathcal{L}_{\sigma}(\bm{y},\bm{s},\bm{z},\bm{r},\bm{v},\bm{o},\bm{q},\bm{t},\bm{x}_1,\bm{x}_2,\bm{x}_3,\bm{x}_4) = \delta^*_{\mathcal{P}_2}(-\bm{o}) + f^*(-\bm{q}) +p^*(-\bm{s}) - \iprod{\bm{b}_1}{\bm{s}} + \frac{1}{2}\iprod{\mathcal{Q}(\bm{v})}{\bm{v}}  \\
& \qquad + \delta_{\mathcal{P}_1}^*(-\bm{t})  + \frac{\sigma}{2}\left(\|\bm{o}-\bm{y}
+  \frac{1}{\sigma}\bm{x}_1 \|_{\mathrm{F}}^2
 + \|\bm{q}-\bm{z}+ \frac{1}{\sigma}\bm{x}_2 \|_{\mathrm{F}}^2 + \|\bm{t}-\bm{r}+ \frac{1}{\sigma}\bm{x}_3 \|^2 \right) \\
& \qquad + \frac{\sigma}{2}(\| \mathcal{A}^*(\bm{y}) + \mathcal{B}^*\bm{z} + \bm{s} - \mathcal{Q}\bm{v} + \bm{r} - \bm{c} + \frac{1}{\sigma}\bm{x}_4 \|_{\mathrm{F}}^2) - \frac{1}{2\sigma}\sum_{i=1}^4\|\bm{x}_i\|^2 .
\end{aligned}
\end{equation*}
Minimizing $\mathcal{L}_{\sigma}$ with respect to the variables $\bm{o},\bm{q},\bm{s},\bm{t}$ yields
\begin{equation}
\begin{aligned}
\bm{o} &= -\text{prox}_{\delta^*_{\mathcal{P}_2}/\sigma}(\bm{x}_1/\sigma - \bm{y}), \quad  \bm{q} = -\text{prox}_{f^*/\sigma}(\bm{x}_2/\sigma -\bm{z}  ), \\
\bm{s} &= -\text{prox}_{p^*/\sigma}(\mathcal{A}^*(\bm{y}) + \mathcal{B}^*\bm{z} - \mathcal{Q}\bm{v} + \bm{r} - \bm{c} + \frac{1}{\sigma}\bm{x}_4 ),
\quad \bm{t}  = -\text{prox}_{\delta^*_{\mathcal{P}_1}/\sigma }(\bm{x}_3/\sigma - \bm{r}).
\end{aligned}
\end{equation}
Let $\bm{w} = (\bm{y},\bm{z},\bm{r},\bm{v},\bm{x}_1,\bm{x}_2,\bm{x}_3,\bm{x}_4)$. Then the modified augmented Lagrangian function is:
\begin{equation} \label{eqn-alm}
\begin{aligned}
\Phi_{\sigma}(\bm{w}) &  =  \underbrace{p^*(\prox_{p^*/\sigma}(\bm{x}_4/\sigma -\mathcal{A}^*(\bm{y}) - \mathcal{B}^*\bm{z}  - \mathcal{Q}\bm{v} -\bm{r} + \bm{c} ) ) + \frac{1}{2\sigma}\|\text{prox}_{\sigma p} (\bm{x}_4 + \sigma(\mathcal{A}^*(\bm{y}) + \mathcal{B}^*\bm{z}  - \mathcal{Q}\bm{v} + \bm{r} - \bm{c})) \|^2}_{ \text{Moreau~envelope~} p^* }
  \\
& \quad   +  \underbrace{\delta_{\mathcal{P}_1}^*(\text{prox}_{\delta^*_{\mathcal{P}_1} }(\bm{x}_3/\sigma - \bm{t} ) ) + \frac{1}{2\sigma}\|\Pi_{\mathcal{P}_1}(\bm{x}_3 -\sigma\bm{r}  ) \|^2}_{ \text{Moreau~envelope~} \delta^*_{\mathcal{P}_1} } +\underbrace{ \delta^*_{\mathcal{P}_2}(\text{prox}_{\delta^*_{\mathcal{P}_2}/\sigma}(\bm{x}_1/\sigma -\bm{y})) + \frac{1}{2\sigma}\|\Pi_{\mathcal{P}_2}(\bm{x}_1 - \sigma \bm{y}) \|^2}_{ \text{Moreau~envelope~} \delta^*_{\mathcal{P}_2}  }     \\
& \quad +  \underbrace{f^*(\text{prox}_{f^*/\sigma}(\bm{x}_2/\sigma -\bm{z} )) + \frac{1}{2\sigma} \|\text{prox}_{\sigma f}(\bm{x}_2 -\sigma \bm{z}  )  \|^2}_{ \text{Moreau~envelope~} f^* } + \frac{1}{2}\iprod{\mathcal{Q}\bm{v}}{\bm{v}} - \frac{1}{2\sigma}\sum_{i=1}^4 \|\bm{x}_i \|^2.
\end{aligned}
\end{equation}
Henceforth, the differentiable saddle point problem is
\begin{equation} \label{prob-minmax}
    \min_{\bm{y},\bm{z},\bm{r},\bm{v}} \max_{\bm{x}_1,\bm{x}_2,\bm{x}_3,\bm{x}_4 } \Phi(\bm{y},\bm{z},\bm{r},\bm{v};\bm{x}_1,\bm{x}_2,\bm{x}_3,\bm{x}_4).
\end{equation}

 In the subsequent analysis, we make the following assumption.
\begin{Assumption}[Slater’s condition] \label{assum}
 The dual problem \eqref{general-dual2} has an optimal solution $\bm{y}_*, \bm{z}_*, \bm{s}_*, \bm{r}_*, \bm{v}_*.$ Furthermore,   Slater’s condition holds for the dual problem \eqref{general-dual}, i.e., there exists $-\bm{y} \in {\rm ri}({\rm dom}(\delta_{\mathcal{P}_2}^*)), - \bm{s} \in {\rm ri}({\rm dom}(p^*)), - \bm{r} \in {\rm dom} (\delta^*_{\mathcal{P}_1}) $  and $ -\bm{z} \in {\rm ri}({\rm dom}(f^*)) $ such that $\mathcal{A}^*(\bm{y}) + \mathcal{B}^*\bm{z} + \bm{s} - \mathcal{Q}\bm{v} + \bm{r} = \bm{c}.$
\end{Assumption}

Based on Slater's condition, the saddle point problem satisfies the strong AL duality.
\begin{lemma}[Strong duality \cite{deng2025augmented}]
    Suppose Assumption \ref{assum} holds. Given any $\sigma > 0$, the strong duality holds for \eqref{prob-minmax}, i.e.,
\begin{equation}\label{lemma:strong}
        \min_{\bm{y},\bm{z},\bm{r},\bm{v}} \max_{\bm{x}_1,\bm{x}_2,\bm{x}_3,\bm{x}_4 } \Phi(\bm{y},\bm{z},\bm{r},\bm{v};\bm{x}_1,\bm{x}_2,\bm{x}_3,\bm{x}_4)=     \max_{\bm{x}_1,\bm{x}_2,\bm{x}_3,\bm{x}_4 }  \min_{\bm{y},\bm{z},\bm{r},\bm{v}} \Phi(\bm{y},\bm{z},\bm{r},\bm{v};\bm{x}_1,\bm{x}_2,\bm{x}_3,\bm{x}_4).
\end{equation}
where both sides of \eqref{lemma:strong} are equivalent to problem \eqref{general}.
\end{lemma}

\subsection{A semismooth Newton method with global convergence} \label{2-2}
It follows from the Moreau envelope theorem \cite{beck2017first} that $e_{\sigma}f^*$, $e_{\sigma}p^*$, $e_{\sigma}\delta_{\mathcal{Q}}^*$, and $e_{\sigma}\delta_{\mathcal{K}}^*$ are continuously differentiable, which implies that $\Phi$ is also continuously differentiable. Hence, the gradient of the saddle point problem can be represented by
\begin{equation} \label{eqn:gradient}
\begin{aligned}
\nabla_{\bm{y}} \Phi_{\sigma}(\bm{w}) & = \mathcal{A} \text{prox}_{\sigma p}(\bm{x}_4 + \sigma(\mathcal{A}^*(\bm{y}) + \mathcal{B}^*\bm{z}- \mathcal{Q}\bm{v} + \bm{r} - \bm{c} ) ) - \Pi_{ \mathcal{P}_2}(\bm{x}_1 - \sigma \bm{y} ), \\
\nabla_{\bm{z}} \Phi_{\sigma}(\bm{w}) & = \mathcal{B} \text{prox}_{\sigma p}(\bm{x}_4 + \sigma(\mathcal{A}^*(\bm{y}) + \mathcal{B}^*\bm{z}- \mathcal{Q}\bm{v} + \bm{r} - \bm{c} ) ) - \text{prox}_{\sigma f}(\bm{x}_2 - \sigma \bm{z}), \\
\nabla_{\bm{r}} \Phi_{\sigma}(\bm{w}) & = \text{prox}_{\sigma p}(\bm{x}_4 + \sigma(\mathcal{A}^*(\bm{y}) + \mathcal{B}^*\bm{z}- \mathcal{Q}\bm{v} + \bm{r} - \bm{c} ) ) - \Pi_{ \mathcal{P}_1}(\bm{x}_3 - \sigma \bm{r} ), \\
\nabla_{\bm{v}} \Phi_{\sigma}(\bm{w}) & = - \mathcal{Q} \text{prox}_{\sigma p}(\bm{x}_4 + \sigma(\mathcal{A}^*(\bm{y}) + \mathcal{B}^*\bm{z}- \mathcal{Q}\bm{v} + \bm{r} - \bm{c} ) ) +
\mathcal{Q}\bm{v}, \\
\nabla_{\bm{x}_1} \Phi_{\sigma}(\bm{w}) & = \frac{1}{\sigma }\Pi_{\mathcal{P}_2}(\bm{x}_1 - \sigma \bm{y} ) - \frac{1}{\sigma} \bm{x}_1, \\
\nabla_{\bm{x}_2} \Phi_{\sigma}(\bm{w}) & = \frac{1}{\sigma }\text{prox}_{\sigma f}(\bm{x}_2 - \sigma \bm{z}) - \frac{1}{\sigma} \bm{x}_2, \\
\nabla_{\bm{x}_3} \Phi_{\sigma}(\bm{w}) & = \frac{1}{\sigma }\Pi_{\mathcal{P}_1}(\bm{x}_3 - \sigma \bm{r} ) - \frac{1}{\sigma} \bm{x}_3, \\
\nabla_{\bm{x}_4} \Phi_{\sigma}(\bm{w}) & = \frac{1}{\sigma }\text{prox}_{\sigma p}(\bm{x}_4 + \sigma(\mathcal{A}^*(\bm{y}) + \mathcal{B}^*\bm{z}- \mathcal{Q}\bm{v} + \bm{r} - \bm{c} ) ) - \frac{1}{\sigma} \bm{x}_4. \\
\end{aligned}
\end{equation}
We note that if $f^*$ is differentiable, $\bm{x}_2$ does not exist and the corresponding gradient is $\nabla_{\bm{z}} \Phi_{\sigma}(\bm{w}) = \mathcal{B} \text{prox}_{\sigma p}(\bm{x}_4 + \sigma(\mathcal{A}^*(\bm{y}) + \mathcal{B}^*\bm{z}- \mathcal{Q}\bm{v} + \bm{r} - \bm{c} ) ) - \nabla f^*(-\bm{z})$.

The nonlinear operator $F(\bm{w})$ is defined as
\begin{small}
\begin{equation} \label{eq:def:F}
     F(\bm{w}) =
     \begin{pmatrix}
          \nabla_{\bm{y}} \Phi(\bm{w});\nabla_{\bm{z}} \Phi(\bm{w});\nabla_{\bm{r}} \Phi(\bm{w});\nabla_{\bm{v}} \Phi(\bm{w});
         - \nabla_{\bm{x}_1} \Phi(\bm{w});
         - \nabla_{\bm{x}_2} \Phi(\bm{w});
         - \nabla_{\bm{x}_3} \Phi(\bm{w});
         - \nabla_{\bm{x}_4} \Phi(\bm{w})
     \end{pmatrix}.
\end{equation}
\end{small}
It is shown in \cite[Lemma 3.1]{deng2025augmented} that $\bm{w}_*$ is a solution of the saddle point problem \eqref{prob-minmax} if and only if it satisfies $F(\bm{w}_*) = 0$. Hence, the saddle point problem can be transformed into solving the following nonlinear equations:
\begin{equation} \label{ssn-eqn}
    F(\bm{w}) =0.
\end{equation}

 \begin{definition} \label{def:Jacobian}
    Let $F$ be a locally Lipschitz continuous mapping. Denote by $D_F$ the set of differentiable points of $F$. The B-Jacobian of $F$ at $\bm{x}$ is defined by
\[
\partial_B F(\bm{w}) := \left\{\lim_{k \rightarrow \infty} J(\bm{w}^k)\, |\,  \bm{w}^k \in D_F, \bm{w}^k \rightarrow \bm{w}\right\},
\]
where $J(\bm{w})$ denotes the Jacobian of $F$ at $\bm{w} \in D_F$. The set $\partial F(\bm{w})$ = $co(\partial_B F(\bm{w}))$ is called the Clarke subdifferential, where $co$ denotes the convex hull.

    $F$ is semismooth  at $\bm{w}$ if $F$ is directionally differentiable at $\bm{w}$ and  for any $\bm{d}$, $J \in \partial F(\bm{w}+\bm{d})$, it holds that
$ \| F(\bm{w}+\bm{d}) -  F(\bm{w}) - J\bm{d} \| = o(\|\bm{d}\|), \;\; \bm{d} \rightarrow 0. $
$F$ is said to be strongly semismooth at $\bm{w}$ if $F$ is directionally differentiable at $\bm{w}$ and
$\| F(\bm{w}+\bm{d}) -  F(\bm{w}) - J\bm{d} \| = O(\|\bm{d}\|^2), \;\; \bm{d} \rightarrow 0.$
We say $F$ is semismooth (strongly semismooth) if $F$ is semismooth (strongly semismooth) for any $\bm{w}$ \cite{mifflin1977semismooth}.
\end{definition}

Note that for a convex function $h$, its proximal operator $\prox_{th}$ is Lipschitz continuous. Then, by Definition \ref{def:Jacobian}, we define the following sets:
 \begin{equation}
 \begin{aligned}
 D _{\Pi_1}  &:=  \partial \Pi_{\mathcal{P}_1} (\bm{x}_3 - \sigma \bm{r}), \;     D _{\Pi_2}  = \partial \Pi_{\mathcal{P}_2}(\bm{x}_1 - \sigma \bm{y}), \;  D_{f}  := \partial  \prox_{\sigma f}( \bm{x}_2 - \sigma \bm{z}), \\
 D_{p} &:=   \partial \prox_{\sigma p}(\bm{x}_4 + \sigma(\mathcal{A}^*(\bm{y}) + \mathcal{B}^*\bm{z}- \mathcal{Q}\bm{v} + \bm{r} - \bm{c} )).
\end{aligned}
 \end{equation}
Hence, the corresponding generalized Jacobian can be represented by
 \begin{equation}\label{equ:jaco}
      \hat{\partial} F(\bm{w}) : = \left\{ \left(
    \begin{array}{cc}
    \mathcal{H}_{\bm{11}}     & \mathcal{H}_{\bm{12}}  \\
    -\mathcal{H}_{\bm{12}}^{\top}     & \mathcal{H}_{\bm{22}}   \\
    \end{array}
    \right)\right\},
\end{equation}
where
\begin{equation} \label{gradient:F}
\begin{aligned}
\mathcal{H}_{\bm{11}} &= \sigma \left( \mathcal{A}, \mathcal{B},\mathcal{I},-\mathcal{Q} \right)^{\mathrm{T}}D_p \left( \mathcal{A}, \mathcal{B},\mathcal{I},-\mathcal{Q} \right) + \sigma \text{blkdiag}(D_{\Pi_1},D_{f},D_{\Pi_2},\mathcal{Q}),\\
\mathcal{H}_{\bm{12}} &=  \left[\left(- \text{blkdiag} \left( [D_{\Pi_1} ,D_{f} ,D_{\Pi_2} ] \right) ; \bm{0} \right) ,(\mathcal{A},\mathcal{B},\mathcal{I},-\mathcal{Q})^{\mathrm{T}}D_{p} \right], \\
\mathcal{H}_{\bm{22}} &= \text{blkdiag}\left\{\frac{1}{\sigma}(\mathcal{I}-D_{\Pi_1}),\frac{1}{\sigma}(\mathcal{I}-D_{h } ),\frac{1}{\sigma}(\mathcal{I}-D_{\Pi_2}),\frac{1}{\sigma}(\mathcal{I}-D_{p}) \right\}.
\end{aligned}
\end{equation}
 It follows from \cite{hiriart1984generalized} and the definition of $ \hat{\partial} F$ that
$ \hat{\partial} F(\bm{w})[\bm{d}] =  \partial F(\bm{w})[\bm{d}]$ for any $\bm{d}$. Hence, $\hat{\partial} F(\bm{w})$ is valid to construct a Newton equation to solve $F(\bm{w}) = 0$.

We next present the semismooth Newton method to solve \eqref{ssn-eqn}.  First,
 an element of the Clarke's generalized Jacobian is taken and defined by \eqref{equ:jaco} as $J^k \in \hat{\partial}F(\bm{w}^k)$. Given $ \tau_{k,i}$, we compute the semismooth Newton direction $\bm{d}^{k,i}$ as the solution of the following linear system
\be \label{eq:ssn} (J^k + \tau_{k,i} \mathcal{I}) \bm{d}^{k, i} = -  F(\bm{w}^k) + \bm{\varepsilon}^k, \ee
where $\bm{\varepsilon}^k$ is the residual term to measure the inexactness of the equation.  We require that there exists a constant $C_{\bm{\varepsilon}} > 0$ such that $\|\bm{\varepsilon}^k\| \le C_{\bm{\varepsilon}} k^{-\beta}$, $\beta \in (1/3 ,1]$. The shift term $\tau_{k,i} \mathcal{I}$ is added to guarantee the existence and uniqueness of $\bm{d}^{k, i}$ and the trial step is defined by
\be \label{eq:ssn-step}
  \bar{\bm{w}}^{k,i}  = \bm{w}^k + \bm{d}^{k,i}.
\ee

Next, we present a globalization scheme to ensure convergence only using regularized semismooth Newton steps. The main idea is to find a suitable $\tau_{k,i}$. It uses both line search on the shift parameter $\tau_{k,i}$ and the nonmonotone decrease on the residuals $F(\bm{w}^k)$. Specifically, for an integer $\zeta\geq 1$, $\nu \in (0,1), \kappa > 1, \gamma > 1, i_{\max} >0, i = 0,\cdots,i_{\max}$, we aim to find the smallest $i$ such that $ \tau_{k,i} = \kappa \gamma^i \|F(\bm{w}^k)\|$ and the nonmonotone decrease condition
\begin{align}
    \|F(\bar{\bm{w}}^{k,i})\|  & \leq  \nu \max_{\max(1, k-\zeta+1) \leq j \leq k}\|F(\bm{w}^j)\| + \varsigma_k \label{eq:decrease-1}
\end{align}
holds,
where $\{ \bm{\varsigma}_i \}$ is a nonnegative sequence such that $ \sum_{i=1}^{\infty} \bm{\varsigma}_i^2 < \infty.$ The iterative update \(\bm{w}^{k+1} = \bar{\bm{w}}^{k,i}\) is performed if condition \eqref{eq:decrease-1} holds. Otherwise, if \eqref{eq:decrease-1} does not hold for $i >i_{\max}$,  we choose $\tau_{k,i}$ such that
\begin{align}
    \tau_{k,i} & \geq c k^{\beta} \label{eq:decrease-2},
  \end{align}
  where $c > 0$ is a given constant and then we set $\bm{w}^{k+1} = \bar{\bm{w}}^{k,i}$.

  Condition \eqref{eq:decrease-1} assesses whether the residuals exhibit a nonmonotone sufficient descent property, which allows for temporary increases in residual values $\|F(\bm{w}^k) \|$. The parameters \(\zeta\) and \(\nu\) govern the number of previous points referenced in this evaluation, where larger values of \(\zeta\) and \(\nu\) lead to more lenient acceptance criteria for the semismooth Newton step. If \eqref{eq:decrease-1} is not satisfied, the regularization parameter \(\tau_{k,i}\) is adjusted according to \eqref{eq:decrease-2}, ensuring a monotonic decrease in the residual sequence \(\{F(\bm{w}^k)\}\) through an implicit mechanism which combines a regularized semismooth Newton step. The nonmonotone strategy provides flexibility by imposing a relatively relaxed condition, which results in the acceptance condition \eqref{eq:decrease-1} with the initial \(\tau_{k,0}\) being satisfied in nearly all iterations, as empirically validated by our numerical experiments.
  The complete procedure is summarized in Algorithm \ref{alg:ssn}.

\begin{algorithm}[htbp]
\caption{A semismooth Newton method for solving \eqref{ssn-eqn}.}
\label{alg:ssn}
\begin{algorithmic}[1]
\Require The constants $\gamma > 1$, $\nu \in (0,1)$, $\beta \in (1/2, 1]$, $\kappa >0$, an integer $\zeta \geq 1$, and an initial point $\bm{w}^0 $, set $k = 0$.
\While {\emph{stopping condition not met}}
\State Compute $F(\bm{w}^k)$ and choose one $J(\bm{w}^k) \in \hat{\partial} F(\bm{w}^k)$.
\State Find the smallest $i \ge 0$ such that $\bar{\bm{w}}^{k,i}$ defined in \eqref{eq:ssn-step} satisfies \eqref{eq:decrease-1} or $\tau_{k,i}$ satisfy \eqref{eq:decrease-2}.
\State Set $\bm{w}^{k+1} = \bar{\bm{w}}^{k,i}$.
\State Set $k=k+1$.
\EndWhile
\end{algorithmic}
\end{algorithm}

We have the following global convergence analysis of Algorithm \ref{alg:ssn} \cite[Theorem 1]{deng2025augmented}.

\begin{theorem} \label{thm:global-con}
Suppose that Assumption \ref{assum} holds. Let $\{\bm{w}^k\}$ be the sequence generated by Algorithm \ref{alg:ssn}. The residual $F(\bm{w}^k)$ converges to $0$, i.e., \be \label{eq:con-w} \lim_{k \rightarrow \infty} \; F(\bm{w}^k) = 0. \ee
\end{theorem}

For local convergence, we first introduce the definition of partial smoothness \cite{lewis2022partial}.
\begin{definition}[$C^p$-partial smoothness]\label{def-psmooth}
Consider a proper closed function $\phi:\R^n\rightarrow \bar{\mathbb{R}}$ and a $C^p$ $(p \ge 2)$ embedded submanifold $\Mcal$ of $\R^n$. The function $\phi$ is said to be $C^p$-partly smooth at $x \in \Mcal$ for $v \in \partial \phi(x)$ relative to $\Mcal$ if
\begin{itemize}
    \item[(i)] Smoothness: $\phi$ restricted to $\mathcal{M}$ is $C^{p}$-smooth near $x$.
    \item[(ii)] Prox-regularity: $\phi$ is prox-regular at $x$ for $v$.

    \item[(iii)] Sharpness: $\mathrm{par}\, \partial_p \phi(x) = N_{\Mcal} (x)$, where $\partial_p $ denotes the set of proximal subgradients of $\phi$ at point $x$, $\mathrm{par}\, \Omega$ is the subspace parallel to $\Omega$, and $N_{\Mcal} (x)$ is the normal space of $\Mcal$ at $x$.
    \item[(iv)] Continuity: There exists a neighborhood $V$ of $v$ such that the set-valued mapping $V \cap \partial \phi$ is inner semicontinuous at $x$ relative to $\mathcal{M}.$
\end{itemize}
\end{definition}

One usage of the partial smoothness is connecting the relative interior condition in (iii) with SC to derive certain smoothness in nonsmooth optimization \cite{bareilles2023newton}.  The local error bound condition \cite{yue2019family} is a powerful tool for analyzing local superlinear convergence in the absence of nonsingularity.
\begin{definition}
We say the local error bound condition holds for $F$ if there exist $\gamma_l > 0$ and $\varepsilon_l > 0$ such that for all $\bm{w}$ with ${\rm dist}(\bm{w},\bm{W}_*)\le \varepsilon_l$, it holds that
\begin{equation}
    \label{eq:eb} \| F(\bm{w}) \| \geq \gamma_{l} {\rm dist}(\bm{w}, \bm{W}_*),
\end{equation}
    where $\bm{W}_*$ is the solution set of $F(\bm{w}) = 0$ and ${\rm dist}(\bm{w}, \bm{W}_*):=\argmin_{\bm{u} \in \bm{W}_*} \|\bm{w} - \bm{u}\|$.
\end{definition}

Using the partial smoothness and local error bound condition,  we have the following local superlinear convergence result \cite[Theorem 2]{deng2025augmented}.

\begin{theorem} \label{thm:local}
Suppose Assumption \ref{assum} holds and $p(\bm{x}), f(\mathcal{B}(\bm{x}))$ are partial smooth. For any optimal solution $\bm{w}_*$, if the SC is satisfied at $\bm{w}_*$, $F$ defined by \eqref{eq:def:F} is locally $C^{p-1}$-smooth
in a neighborhood of $\bm{w}_*$. Furthermore, if $\bm{w}^k$ is close enough to $\bm{w}_* \in \bm{W}^*$ where the SC and the local error bound condition \eqref{eq:eb} hold, then \eqref{eq:decrease-1} always holds with $i= 0$ and $\bm{w}^k$ converges to $\bm{w}_*$ Q-superlinearly.
\end{theorem}

 Notably, the partial smoothness and Slater's condition are commonly encountered in various applications.
    Even though the local error bound condition may appear restrictive, such a condition is satisfied when the functions $p$ and $f$ are piecewise linear-quadratic, such as $\ell_1,\ell_{\infty}$ norm and box constraint.
\subsection{An efficient implementation to solve the linear system} \label{2-3}
Ignoring the subscript $k$, the linear system \eqref{eq:ssn} can be represented by:
\begin{equation}\label{eq:ssn-xz}
\left(
    \begin{array}{cc}
    \mathcal{H}_{\bm{11}} + \tau \mathcal{I}     & \mathcal{H}_{\bm{12}}  \\
    -\mathcal{H}_{\bm{12}}^T     & \mathcal{H}_{\bm{22}} + \tau \mathcal{I} \\
    \end{array}
    \right) \left(
    \begin{array}{c}
    \bm{d}_{\bm{1}}       \\
    \bm{d}_{\bm{2}}       \\
    \end{array}
    \right) = -\left(
    \begin{array}{c}
    \tilde{F}_{\bm{1}}       \\
    \tilde{F}_{\bm{2}}    \\
    \end{array}
    \right),
\end{equation}
where $\tilde{F} = F - \bm{\varepsilon}, F = (F_1,F_2)$,  $F_1 = ( F_{\bm{y}},F_{\bm{z}}, F_{\bm{r}},F_{\bm{v}} ), F_2 = (F_{\bm{x}_1},F_{\bm{x}_2}, F_{\bm{x}_3}, F_{\bm{x}_4} ),  \bm{d}_1 = (\bm{d}_{\bm{y}},\bm{d}_{\bm{z}},\bm{d}_{\bm{r}},\bm{d}_{\bm{v}} ), \bm{d}_2 =( \bm{d}_{\bm{x}_1},\bm{d}_{\bm{x}_2},\bm{d}_{\bm{x}_3},\bm{d}_{\bm{x}_4} ).$
For a given $\bm{d}_{\bm{1}}$, the direction $\bm{d}_{\boldsymbol{2}}$ can be calculated by
 \begin{equation} \label{cor:z}
 \bm{d}_{\bm{2}} = (\mathcal{H}_{\bm{22}} + \tau \mathcal{I})^{-1}(\mathcal{H}_{\bm{12}}^{\top}\bm{d}_{\bm{1}}-F_{\bm{2}}).
  \end{equation}
Hence, the linear equation \eqref{eq:ssn-xz} reduces to a  linear system with respect  to $\bm{d}_{\bm{1}}$:
\begin{equation} \label{eqn:simp}
    \widetilde{\mathcal{H}}_{\bm{11}} \bm{d}_{\bm{1}} = -\widetilde{F}_{\bm{1}},
\end{equation}
where  $\widetilde{F}_{\bm{1}}:=\mathcal{H}_{\bm{12}} (\mathcal{H}_{\bm{22}}+\tau \mathcal{I})^{-1} \tilde{F}_{\bm{2}} - \tilde{F}_{\bm{1}}$ and $\widetilde{\mathcal{H}}_{\bm{11}}:= (\mathcal{H}_{\bm{11}} +
 \mathcal{H}_{\bm{12}}(\mathcal{H}_{\bm{22}} + \tau \mathcal{I})^{-1}\mathcal{H}_{\bm{12}}^{\top} + \tau \mathcal{I}).$ The definition of  $\mathcal{H}_{\bm{12}}$ in \eqref{equ:jaco} yields
\begin{equation} \label{eqn:equation}
\widetilde{\mathcal{H}}_{\bm{11}} =
  \left( \mathcal{A}, \mathcal{B},\mathcal{I},\mathcal{Q} \right)^{\mathrm{T}} \overline{D}_p \left( \mathcal{A}, \mathcal{B},\mathcal{I},\mathcal{Q} \right) + \sigma \text{blkdiag}(\overline{D}_{\Pi_1},\overline{D}_{\mathrm{F}},\overline{D}_{\Pi_2}, \mathcal{Q}),
\end{equation}
      where $\text{blkdiag}$ denotes the block diagonal operator, $ \overline{D}_{p} = \sigma D_{p} + \widetilde{D}_{p}, \widetilde{D}_{p} = D_{p} (\frac{1}{\sigma}(\mathcal{I} - D_{p}) + \tau \mathcal{I})^{-1} D_{p}, \overline{D}_{\Pi_1},\overline{D}_{\mathrm{F}}$ and $\overline{D}_{\Pi_2}$ are defined analogously. If the problem has more than one primal variable, we can solve the linear system \eqref{eqn:simp} using iterative methods.
      According to \eqref{eqn:equation}, $\left( \mathcal{A}, \mathcal{B},\mathcal{I},\mathcal{Q} \right) \bm{d}_1$ can be computed first and shared among all components. If the corresponding solution is sparse or low-rank, then the special structures of $\overline{D}_{p}$ can further be used to improve the computational efficiency. Furthermore, if $\widetilde{\mathcal{H}}_{11}$ only has one variable, we can solve the equation \eqref{eqn:simp} using direct methods, such as Cholesky factorization method.

We also note that some variables in $\bm{w}$ may not exist if the function or constraint does not exist in \eqref{general}.  The existence condition of variables is listed in the following.
\begin{itemize}
    \item $\bm{y}$ exists if and only if $\mathcal{P}_2$ is nontrivial. $\bm{x}_1$ exist if and only if $\mathcal{P}_2$ is not a singleton set.
    \item $\bm{z}$ exists if and only if $f$ exists. $\bm{x}_2$ exists if and only if $f$ exists and is nonsmooth.
    \item $\bm{r}$ and $\bm{x}_3$ exist if and only if $\mathcal{P}_1$ is nontrivial.
    \item $\bm{v}$ exists if and only if $\mathcal{Q}$ is nontrivial.
\end{itemize}
For example, for the Lasso problem, $p(\bm{x}) = \|\bm{x}\|_1, f(\mathcal{B}(\bm{x})) = \frac{1}{2} \|\mathcal{B}(\bm{x}) - \bm{b}\|^2, \bm{c} = \bm{0}, \mathcal{Q} = \bm{0}, \mathcal{P}_1 = \mathcal{P}_2 = \oslash.$ The valid variables are $\bm{z}$ and $\bm{x}_4$, i.e., one primal variable and one dual variable.
Consequently, for problems where $\mathcal{H}_{11}$ only has one primal variable, such as Lasso, and SOCP, we can solve the linear system using direct methods such as Cholesky factorization at low cost.


\subsection{Practical implementations} \label{2-4}
 To ensure that Algorithm \ref{alg:ssn} has a better performance on various problems, we present some implementation details of Algorithm \ref{alg:ssn} used to solve \eqref{ssn-eqn} in this section.
\subsubsection{Line search for $\bm{d}^{k}$}
In some cases, condition \eqref{eq:decrease-1} may not be satisfied with the full regularized Newton step in \eqref{eq:ssn-step}.
 The sufficient decrease property \eqref{eq:decrease-1} may be easier to satisfy when a line search strategy is used for problems such as Lasso-type problems. Specifically, we choose appropriate $\alpha$ and $ \tilde{\bm{w}}^{k,i}  = \bm{w}^k + \alpha \bm{d}^{k,i}$ such that condition
 \begin{equation} \label{eq:decrease-11}
  \|F(\tilde{\bm{w}}^{k,i})\|   <  \nu \max_{\max(1, k-\zeta+1)} \|F(\bm{w}^j)\| + \varsigma_k
 \end{equation}
 holds, we then set $\bm{w}^{k+1} = \tilde{\bm{w}}^{k,i}.$  If \eqref{eq:decrease-11} is not satisfied after several line searches, then we set $\bm{w}^{k+1} = \bar{\bm{w}}^{k, i}$ with \eqref{eq:decrease-2} being held.  Since it needs one additional proximal operator calculation every time, the line search property is only effective for $p(x)$ whose proximal operator can be calculated efficiently.
\subsubsection{Update regularization parameter $\kappa$} $\kappa$ serves as the constant in the definition of $\tau_{k,i}$ when $i < i_{\max}$, which is of vital importance to control the quality of $\bm{w}^k.$
When
$\kappa$ is small, the Newton equation is accurate, but $\bm{d}^k$ may not be a good direction. For an iterate $\bm{w}^k$, $\bm{d}_1^k$ and $\bm{d}_2^k$ are descent or ascent directions if for the corresponding primal and dual variables if $\iprod{\bm{d}_1^k}{F_1} < 0$ and $\iprod{\bm{d}_2^k}{F_2} < 0$, respectively. Taking into account this situation, we define the ratio
\begin{equation} \label{ratio}
    \rho_k : = \frac{ -\iprod{\bm{d}^k}{F(\bm{w}^{k+1})}}{\|\bm{d}^k\|_2^2}
\end{equation}
to decide whether $\bm{d}_k$ is a bad direction and how to update $\kappa_k.$ If $\rho_k$ is small, it is usually a signal of a bad Newton step and we increase $\kappa_k$. Otherwise, we decrease it. Specifically, the parameter $\kappa_k$ is updated as
\begin{equation}
    \kappa_{k+1} = \begin{cases}
        \max\{\gamma_1\kappa_k,\underline{\tau}\}, & \mbox{if}\, \rho_k \ge \eta_2,  \\
        \gamma_2 \kappa_k, & \mbox{if}\,  \eta_2 > \rho_k \ge \eta_1, \\
        \min\{\gamma_3\kappa_k,\bar{\tau}\},  & \mbox{otherwise},
    \end{cases}
\end{equation}
where $0 <\eta_1 \le \eta_2, 0 < \gamma_1 < \gamma_2 < 1, \gamma_3 > 1$ are chosen parameters and $\underline{\tau}, \bar{\tau} $  are two predefined positive constants.
\subsubsection{Update penalty parameter $\sigma$}
We also adaptively adjust the penalty factor $\sigma$ based on the primal and dual infeasibility. Specifically, if the primal infeasibility exceeds the dual infeasibility over a certain number of steps, we decrease $\sigma$; otherwise, we increase it. Specifically, we next show our strategies for how to update $\sigma$ incorporating the iteration information. We mainly examine the ratios of primal and dual infeasibilities of the last few steps defined by
\begin{equation} \label{ratio-adaptive}
 \omega^k = \frac{\text{geomean}_{k-l \leq j \leq k} \eta_P^j}{\text{geomean}_{k-l \leq j \leq k } \eta_D^j},
\end{equation}
where the primal infeasibility $\eta_P$ and the dual infeasibility $\eta_D$ are defined by
\begin{equation}
\eta_{P} = \frac{\| \mathcal{A}(\bm{x}) - \Pi_{\mathcal{P}_2}(\mathcal{A}(\bm{x}) - \bm{y} ) \|  }{1+\|
\bm{x}\|} ,\quad \text{and} \quad
 \eta_D = \frac{\| \mathcal{A}^*(\bm{y}) + \mathcal{B}^*(\bm{z}) + \bm{s} - \mathcal{Q}(\bm{v})- \bm{c}\|  }{1 + \|\bm{c}\|},
\end{equation}
and $l$ is a hyperparameter. For every $l$ steps, we check $\omega^k$. If $\omega^k$ is larger (or smaller) than a constant $\delta$, we decrease (or increase) the penalty parameter $\sigma$ by a multiplicative factor $\gamma$ (or $1/\gamma$) with $0 < \gamma < 1$. To prevent $\sigma$ from becoming excessively large or small, upper and lower bounds are imposed on $\sigma$. This strategy has been demonstrated to be effective in solving SDP problems \cite{li2018semismooth}.


\section{Properties of proximal operators} \label{3}
In this section, we demonstrate how to handle the shift term and the computational details of other proximal operators.
According to \eqref{cor:z}, we need the explicit calculation process of $(\frac{1}{\sigma}(\mathcal{I} - D_p) + \tau \mathcal{I})^{-1}$ and $\sigma D_p+  D_p(\frac{1}{\sigma}(\mathcal{I} - D_p) + \tau \mathcal{I})^{-1}D_p $. Furthermore, if $\bm{x}$ and $\mathcal{B}(\bm{x})$  are replaced by $\bm{x} - \bm{b}_1$ or $\mathcal{B}(\bm{x}) - \bm{b}_2$, respectively, the variables need to be corrected by a shift term. Some proximal operators, such as the semidefinite cone or the $\ell_{\infty}$ norm, are already known in the literature.

 \subsection{Handling \texttt{shift} term}
For problems that have a shift term such as $p(\bm{x} - \bm{b}_1)$ or $f(\mathcal{B}(\bm{x}) - \bm{b}_2)$, the corresponding dual problem of \eqref{general} is
\begin{equation} \label{dual2}
    \begin{aligned}
\min_{\bm{y},\bm{z},\bm{s},\bm{r},\bm{v}} \quad & \delta_{\mathcal{P}_2}^*(-\bm{o}) + f^*(\bm{-q}) - \iprod{\bm{b}_2}{\bm{q}} - \iprod{\bm{b}_1}{\bm{s}} +  p^*(-\bm{s}) + \frac{1}{2} \iprod{\mathcal{Q}\bm{v}}{\bm{v}} + \delta_{\mathcal{P}_1}^*(-\bm{t}), \\
    \st \quad & \mathcal{A}^*(\bm{y}) + \mathcal{B}^*\bm{z} + \bm{s} - \mathcal{Q}\bm{v} + \bm{r} = \bm{c},~~\bm{y}=\bm{o},~~\bm{z}=\bm{q},~~\bm{r}=\bm{t}.
    \end{aligned}
\end{equation}
If $f^*$ is differentiable, the gradient with respect to $\bm{q}$ is $-\nabla f^*(-\bm{q}) - \bm{b}_2$. If $f$ is nonsmooth, it follows from the property of the proximal operator that $\bm{q} = \text{prox}_{f^*/\sigma}(\bm{x}_2/\sigma - \bm{z} - \bm{b}_2/\sigma)$. Hence, we only need to replace the $\text{prox}_{\sigma f}(\bm{x}_2 - \sigma \bm{z})$ in \eqref{eqn:gradient} with $\text{prox}_{\sigma f} (\bm{x}_2 - \sigma (\bm{z} - \bm{b}_2) ) - \bm{b}_2.$ Similarly, for $p^*(-\bm{s})$, the corresponding term $\text{prox}_{\sigma p}(\bm{x}_4 + \sigma(\mathcal{A}^*(\bm{y}) + \mathcal{B}^*\bm{z}- \mathcal{Q}\bm{v} + \bm{r} - \bm{c} ) )$ is replaced by $\text{prox}_{\sigma p}(\bm{x}_4 + \sigma(\mathcal{A}^*(\bm{y}) + \mathcal{B}^*\bm{z}- \mathcal{Q}\bm{v} + \bm{r} - \bm{c} - \bm{b}_1) ) + \bm{b}_1$. Hence, we do not need to introduce a slack variable when adding a shift term $\bm{b}_1$ or $\bm{b}_2$ to $p(\bm{x})$ or $f(\mathcal{B}(\bm{x}))$.
\subsection{ $\ell_2$ norm regularizer}

For the $\ell_2$ norm, i.e., $p(\bm{x}) = \lambda \|\bm{x}\|_2$, its proximal operator is
$
\prox_{\lambda \|\cdot \|_2}(\bm{x}) =  \begin{cases}
        \bm{x} - \lambda \bm{x} /\|\bm{x}\|, & \mbox{if } \|\bm{x}\| > \lambda, \\
        0, & \mbox{otherwise}.
      \end{cases}.
$
Consequently, one generalized Jacboian $D$ of the $\ell_2$ norm is
\[
D =
\begin{cases}
     I -\lambda (I - \frac{\bm{x}\bm{x}^{\mathrm{T}} }{\| \bm{x}\|^2})/\|\bm{x} \|, & \text{if}~~ \|\bm{x} \| > \lambda,\\
     0,  & \text{otherwise},
\end{cases}
\]
It follows from the SMW formula
$
(A - {uu}^{\mathrm{T}})^{-1} = A^{-1} + \frac{A^{-1 }{uu}^{\mathrm{T}}A^{-1} }{1 - {u}^{\mathrm{T}}A^{-1}{u} }
$
that for $D \in \partial \text{prox}_{\lambda \|\cdot \|_2} (\bm{x})$,
\[
\begin{aligned}
 \left(\tau I + \frac{1}{\sigma}(I - D) \right)^{-1} & =  \left(\tau I  + \lambda(I - \bm{x}\bm{x}^{\mathrm{T}}/\|\bm{x}\|^2 )/(\sigma\|\bm{x}\|)\right)^{-1}  = \frac{1}{\tau + \frac{\lambda }{\sigma \|\bm{x}\|}} \left(I + \frac{\lambda}{\sigma \tau  \|\bm{x}\|} \bm{x}^{\mathrm{T}} \bm{x} / \|\bm{x}\|^2 \right).
\end{aligned}
\]
Hence, the following qualities hold:
\begin{equation} \label{eqn:l2op}
\begin{aligned}
D\left(\tau I + \frac{1}{\sigma}(I - D) \right)^{-1} &= \left( \left(1 - \frac{\lambda }{\|\bm{x}\|} \right)I + \left(\frac{1}{\tau\sigma} + 1\right) \frac{\lambda}{\|\bm{x}\|} \bm{x}\bm{x}^{\mathrm{T}}/\|\bm{x}\|^2 \right) \left( \frac{1}{\tau + \frac{\lambda }{\sigma\|\bm{x}\|} } \right) ,\\
\overline{D} = \sigma D + D\left(\tau I + \frac{1}{\sigma}(I - D) \right)^{-1}D &=  \left( \sigma \left(1 - \frac{ \lambda}{\|\bm{x}\|}\right) + \left( \frac{1}{\tau + \frac{\lambda}{\sigma\|\bm{x}\|} } \right) \left(1 - \frac{ \lambda}{\|\bm{x}\|} \right) \right) I\\
& \qquad + \left( \left( \frac{1}{\tau + \frac{\lambda}{\sigma \|\bm{x}\|} } \right) \left(2 \frac{ \lambda}{\|\bm{x}\|} + \frac{ \lambda }{\|\bm{x}\| \sigma \tau } - \frac{\lambda^2}{\|\bm{x}\|^2}\right) +  \frac{\sigma \lambda}{\|\bm{x}\|}\right)\bm{x}\bm{x}^{\mathrm{T}}/\|\bm{x}\|^2.
\end{aligned}
\end{equation}
Consequently, the operators in \eqref{eqn:l2op} can be represented as an identity matrix multiplied by a constant plus a rank-one correction.
 For $p(\bm{x}) = \delta_{\|\bm{x} \|_2 \le \lambda}(\bm{x})$, the derivation is similar and omitted.
\subsection{Second-order cone}
Let \( Q^n \subseteq \mathbb{R}^n \) denote the n-dimensional second-order cone (SOC), defined as
$$ Q^n = \left\{ \bm{x} \in \mathbb{R}^n \,:\, \| \bar{\bm{x}} \| \leq x_0 \right\}. $$
Here, a vector \( \bm{x} \in \mathbb{R}^n \) is partitioned as \( \bm{x} = [x_0; \bar{\bm{x}}] \), where \( x_0 \in \mathbb{R} \) is its scalar part and \( \bar{\bm{x}} \in \mathbb{R}^{n-1} \) is its vector part. For any \( \bm{x} \) in the interior of the cone, \( \bm{x} \in \text{int}(Q^n) \), its determinant is given by \( \det(\bm{x}) = x_0^2 - \| \bar{\bm{x}} \|^2 \). If the determinant is non-zero, its inverse is $\bm{x}^{-1} = \frac{1}{\det(\bm{x})} [x_0; -\bar{\bm{x}}]$. A generalized Jacobian $D$ associated with the second-order cone is given by:

\begin{equation} \label{genJacboian:soc}
D =
\begin{cases}
I, & \text{if } x_0 \ge \| \bar{x} \|, \\
0, & \text{if } x_0 \le -\| \bar{x} \|, \\
\frac{1}{2} \begin{bmatrix}
1 & \frac{\bar{x}^\top}{\|\bar{x}\|} \\
\frac{\bar{x}}{\|\bar{x}\|} & \left(1 + \frac{x_0}{\|\bar{x}\|}\right)I - \frac{x_0}{\|\bar{x}\|^3}\bar{x}\bar{x}^\top
\end{bmatrix}, & \text{otherwise.}
\end{cases}
\end{equation}
For the third case, the generalized Jacobian of the second-order cone admits a low-rank decomposition:
\[
\begin{aligned}
D = \frac{1}{2} \left(1 + \frac{x_0}{\|\bar{x}\|}\right) I + \frac{1}{2} \left(1 - \frac{x_0}{\|\bar{x}\|}\right)
\begin{bmatrix}
\frac{\sqrt{2}}{2} \\
\frac{\sqrt{2}}{2} \frac{\bar{x}}{\|\bar{x}\|}
\end{bmatrix}
\begin{bmatrix}
\frac{\sqrt{2}}{2} & \frac{\sqrt{2}}{2} \frac{\bar{x}^T}{\|\bar{x}\|}
\end{bmatrix}
+ \left(-\frac{1}{2}\right) \left(1 + \frac{x_0}{\|\bar{x}\|}\right)
\begin{bmatrix}
\frac{\sqrt{2}}{2} \\
-\frac{\sqrt{2}}{2} \frac{\bar{x}}{\|\bar{x}\|}
\end{bmatrix}
\begin{bmatrix}
\frac{\sqrt{2}}{2} & -\frac{\sqrt{2}}{2} \frac{\bar{x}^T}{\|\bar{x}\|}
\end{bmatrix}.
\end{aligned}
\]
Define the logarithmic barrier function for any \( x \in Q^n \) by
\( g : \mathbb{R}^n \mapsto \mathbb{R} \) with
\[
g(x) =
\begin{cases}
-\frac{1}{2} \log(\det(x)), & \text{if } \| x_0  \| > \bar{x}, \\
+\infty, & \text{otherwise.}
\end{cases}
\]
We note that $\lim_{\mu \rightarrow 0} \mu g(x) = \delta_{Q^n}(x)$.
For SOCP, dealing with the smooth barrier functions may yield better numerical results than dealing with the conic constraints directly. In this case, we use a barrier function $p(\bm{x}) = \mu g(\bm{x})$ to replace the cone constraint function $p(\bm{x}) = \delta_{Q}(\bm{x})$, where $\mu >0$.  For the smoothing function $\mu g(x)$, the following lemma holds.
\begin{lemma} \label{lem:socp}

(i) The proximal mapping of \(\mu g(x)\) is given by \( \text{prox}_{\mu g} : \mathbb{R}^n \mapsto \text{int}(\mathcal{Q}^n) \) with
\begin{equation} \label{lem3-1-0}
\emph{prox}_{\mu g}(z) =
\begin{bmatrix}
\frac{1}{2}\left(z_0 + \sqrt{\frac{1}{2}(\|z\|^2 + 4\mu + \Delta)}\right) \\
\frac{\bar{z}}{2}\left(1 + \frac{\sqrt{2}z_0}{\sqrt{\|z\|^2 + 4\mu + \Delta}}\right)
\end{bmatrix}, \quad z \in \mathbb{R}^n,
\end{equation}
where \(\Delta = \sqrt{\det(z)^2 + 8\mu \|z\|^2 + 16\mu^2}\). Furthermore, the inverse function of the proximal mapping is given by \( \text{prox}_{\mu g}^{-1} : \text{int}(\mathcal{Q}^n) \to \mathbb{R} \) with
\begin{equation} \label{lem3-1-1}
\emph{prox}_{\mu g}^{-1}(x) = x - \mu x^{-1}, \quad x \in \text{int}(\mathcal{Q}^n).
\end{equation}

(ii) The projection function is the limit of the proximal mapping as \(\mu\) approaches 0, i.e.,
\[
\lim_{\mu \to 0} \text{prox}_{\mu g}(z) = \Pi_K(z), \quad z \in \mathbb{R}^n.
\]
(iii) For \(z \in \mathbb{R}^n\), let \(x = \text{prox}_{\mu g}(z)\). The inverse matrix of the derivative of the proximal mapping at the point \(z\) is given by

\[
(\partial_z \emph{prox}_{\mu g}(z))^{-1} = I - \mu \partial_x(x^{-1}),
\]

where \(\partial_x(x^{-1}) = \frac{1}{\det(x)} \begin{bmatrix}
1 & \\
& -I_{n-1}
\end{bmatrix} - 2(x^{-1})(x^{-1})^\top\).

(iv) The derivative of the proximal mapping at the point \(z\) is given by
\[
\begin{aligned}
\partial_z \emph{prox}_{\mu g}(z) &=
\begin{bmatrix}
\frac{\det(x)}{\det(x)-\mu} &  \\
 & \frac{\det(x)}{\det(x)+\mu} I_{n-1}
\end{bmatrix}
- \frac{2\mu \det(x) \begin{bmatrix}
\frac{x_0}{\det(x)-\mu}  \\
\frac{-\bar{x}}{\det(x)+\mu}
\end{bmatrix}
\begin{bmatrix}
 \frac{x_0}{\det(x)+\mu} \\
\frac{-\bar{x}}{\det(x)+\mu}
\end{bmatrix}^{\mathrm{T}} }{\det(x) + 2\mu \left( \frac{x_0^2}{\det(x)-\mu} + \frac{\|\tilde{x}\|^2}{\det(x)+\mu} \right)} := \Lambda + a {u} {u}^{\mathrm{T}},
\end{aligned}
\]
where $\Lambda = \begin{bmatrix}
    a_0 &  \\
    & a_1 I_{n-1}
\end{bmatrix} $ and $a_0,a_1 \in \mathbb{R}$ are constants.
\end{lemma}

\begin{proof}
(i) Given $ z\in \mathbb{R}^{n} $, it follows from the definition of the proximal mapping that $ x=\operatorname{prox}_{\mu g}(z) $ is the optimal point of the following minimization problem
$$ \min_{x\in \operatorname{int}(\mathcal{Q}^n)}f(x)=\frac{1}{2}\|z-x\|^{2}-\frac{1}{2}\mu\log\det(x). $$
$ x\in \operatorname{int}(\mathcal{Q}^n) $ implies $ \det(x)>0. $ The optimality condition $x-z-\mu x^{-1}=0 $ is equivalent to

\begin{equation} \label{lemx0}
\text{prox}_{\mu g}^{-1}(x) = z=x-\mu x^{-1}\Longleftrightarrow\begin{cases}z_{0}=x_{0}-\frac{\mu}{\det(x)}x_{0},\\
\bar{z}=\bar{x}+\frac{\mu}{\det(x)}\bar{x}.\end{cases}
\end{equation}
Hence \eqref{lem3-1-1} holds.  To derive the expression of $ \operatorname{prox}_{\mu g}(z) $, we consider the following two cases. If $ \frac{\mu}{\det(x)}=1 $, then $ z_{0}=0,\bar{x}=\frac{1}{2}\bar{z} $ and $ x_{0}=\sqrt{\det(x)+\|\bar{x}\|^{2}}=\sqrt{\frac{1}{4}\|\bar{z}\|^{2}+\mu} $. Otherwise $\frac{\mu}{\det(x)}\neq 1$, then $z_{0}\neq 0$. Provided with this condition, $\nabla f(x)=0$ is equivalent to
\begin{equation}
x_{0}=\frac{z_{0}}{1-\rho},\quad\bar{x}=\frac{\bar{z}}{1+\rho},
\end{equation}
where $\rho=\frac{\mu}{\det(x)}>0$. Combined with the identity that $\det(x)=x_{0}^{2}-\|\bar{x}\|^{2}$, we see $\rho$ is the root of the polynomial equation
$$
\frac{z_{0}^{2}}{(1-\rho)^{2}}-\frac{\|\bar{z}\|^{2}}{(1+\rho)^{2}}=\frac{\mu}{\rho}\Longleftrightarrow\frac{z_{0}^{2}}{\rho+\frac{1}{\rho}-2}-\frac{\|\bar{z}\|^{2}}{\rho+\frac{1}{\rho}+2}=\mu.
$$

Let $r=\rho+\frac{1}{\rho}$. Note that $r>2$. By solving the above equation, we have
$$
r=\frac{\det(z)+\Delta}{2\mu},\quad\rho=\begin{cases}
\frac{r-\sqrt{r^{2}-4}}{2} & \text{if } z_{0}>0,\\
\frac{r+\sqrt{r^{2}-4}}{2} & \text{if } z_{0}<0,
\end{cases}
$$
where $\Delta=\sqrt{\det(z)^{2}+8\mu\|z\|^{2}+16\mu^{2}}$. Subsequently, we take $\rho$ into \eqref{lemx0} and have for $z_{0}>0$, $\rho<1$,
\begin{align*}
x_{0}&=\frac{z_{0}}{1-\rho}=\frac{z_{0}}{2}\left(1+\sqrt{\frac{r+2}{r-2}}\right)=\frac{1}{2}\left(z_{0}+\sqrt{\frac{1}{2}(\|z\|^{2}+4\mu+\Delta)}\right),\\
\bar{x}&=\frac{\bar{z}}{1+\rho}=\frac{\bar{z}}{2}\left(1+\sqrt{\frac{r-2}{r+2}}\right)=\frac{\bar{z}}{2}\left(1+\frac{\sqrt{2}z_{0}}{\sqrt{\|z\|^{2}+4\mu+\Delta}}\right).
\end{align*}
For $z_{0}<0$, $\rho>1$,
\begin{align*}
x_{0}&=\frac{z_{0}}{1-\rho}=\frac{z_{0}}{2}\left(1-\sqrt{\frac{r+2}{r-2}}\right)=\frac{1}{2}\left(z_{0}+\sqrt{\frac{1}{2}(\|z\|^{2}+4\mu+\Delta)}\right),\\
\bar{x}&=\frac{\bar{z}}{1+\rho}=\frac{\bar{z}}{2}\left(1-\sqrt{\frac{r-2}{r+2}}\right)=\frac{\bar{z}}{2}\left(1+\frac{\sqrt{2}z_{0}}{\sqrt{\|z\|^{2}+4\mu+\Delta}}\right).
\end{align*}
Therefore \eqref{lem3-1-0} holds for any $z\in\mathbb{R}^{n}$.

(ii) Let $\mu \to 0$, then $\Delta \to |\det(z)|$. Hence, it follows that
$$
\begin{aligned}
\lim_{\mu \to 0} \operatorname{prox}_{\mu g}(z) &=
\begin{bmatrix}
\frac{1}{2}\left(z_0 + \sqrt{\frac{1}{2}(\|z\|^{2}+|\det(z)|)}\right) \\[8pt]
\frac{\bar{z}}{2}\left(1+\dfrac{\sqrt{2}z_0}{\sqrt{\|z\|^{2}+|\det(z)|}}\right)
\end{bmatrix} =
\begin{cases}
z, & \text{if } z_0 \ge \|\bar{z}\|,\\
0, & \text{if } z_0 \le -\|\bar{z}\|,\\
\begin{bmatrix}
\frac{1}{2}(z_0 + \|\bar{z}\|) \\[8pt]
\frac{\bar{z}}{2}\left(1+\dfrac{z_0}{\|\bar{z}\|}\right)
\end{bmatrix}, & \text{if } - \|\bar{z}\| < z_0 < \|\bar{z}\|
\end{cases} = \Pi_{K}(z), \quad z \in \mathbb{R}^{n}.
\end{aligned}
$$

(iii) Note that $\operatorname{prox}_{\mu g}^{-1}$ is a single-valued mapping. By the inverse function theorem, it holds that
\begin{align*}
(\partial_z\operatorname{prox}_{\mu g}(z))^{-1} &= \partial_x\operatorname{prox}_{\mu g}^{-1}(x) = \partial_x \left( x - \mu x^{-1} \right) \overset{\eqref{lemx0}}{=} I - \mu\partial_x(x^{-1}) = I - \mu\left(\frac{1}{\det(x)}\begin{bmatrix}1 & \\ & -I_{n-1}\end{bmatrix} - 2(x^{-1})(x^{-1})^\top\right),
\end{align*}
where the last equation is obtained from the following derivation:
\begin{align*}
\partial_x \left( x^{-1} \right)
  &= \partial_x \left(\frac{1}{\det(x)}
        \begin{bmatrix}x_0\\-\bar{x}\end{bmatrix}\right) = -\frac{1}{\det(x)^2}
        \partial_x\left( \det(x) \right)
        \begin{bmatrix}x_0\\-\bar{x}\end{bmatrix}^{\top}
     +\frac{1}{\det(x)}
        \begin{bmatrix}
          1 &  \\
           & -I_{n-1}
        \end{bmatrix} \\
  &= -\frac{2}{\det(x)^2}
        \begin{bmatrix}x_0\\-\bar{x}\end{bmatrix}
        \begin{bmatrix}x_0\\-\bar{x}\end{bmatrix}^{\!\top}
     +\frac{1}{\det(x)}
        \begin{bmatrix}
          1 &  \\
           & -I_{n-1}
        \end{bmatrix} = -2\left( x^{-1} \right) \left(x^{-1}\right)^{\top}
     +\frac{1}{\det(x)}
        \begin{bmatrix}
          1 &  \\
           & -I_{n-1}
        \end{bmatrix}.
\end{align*}

(iv) Let $\rho = \frac{\mu}{\det(x)}$, $\Lambda = I - \rho\begin{bmatrix}1 & \\ & -I_{n-1}\end{bmatrix}$, $v = x^{-1}$. Then $(\partial_z\operatorname{prox}_{\mu g}(z))^{-1} = \Lambda + 2\mu vv^\top$. By the SMW formula, we have
\begin{align*}
\partial_z\operatorname{prox}_{\mu g}(z) &= \Lambda^{-1} - \frac{2\mu\Lambda^{-1}vv^\top\Lambda^{-1}}{1 + 2\mu v^\top\Lambda^{-1}v} = \begin{bmatrix}\frac{1}{1-\rho} & \\ & \frac{1}{1+\rho}I_{n-1}\end{bmatrix}
- \frac{2\mu\begin{bmatrix}\frac{x_0}{1-\rho} \\ \frac{-\bar{x}}{1+\rho}\end{bmatrix}\begin{bmatrix}\frac{x_0}{1-\rho} \\ \frac{-\bar{x}}{1+\rho}\end{bmatrix}^\top}{\det(x)^2 + 2\mu\left(\frac{x_0^2}{1-\rho} + \frac{\|\bar{x}\|^2}{1+\rho}\right)} \\
&= \begin{bmatrix}\frac{1}{1-\rho} & \\ & \frac{1}{1+\rho}I_{n-1}\end{bmatrix}
- \frac{2\begin{bmatrix}\frac{x_0}{1-\rho}\\  \frac{-\bar{x}}{1+\rho}\end{bmatrix}\begin{bmatrix}\frac{x_0}{1-\rho} \\ \frac{-\bar{x}}{1+\rho}\end{bmatrix}^\top}{\frac{\det(x)}{\rho} + 2\left(\frac{x_0^2}{1-\rho} + \frac{\|\bar{x}\|^2}{1+\rho}\right)},
\end{align*}
It follows from $\rho = 1$ or $\rho \to 1$ that
$$
\partial_z\operatorname{prox}_{\mu g}(z) = \frac{1}{2}\begin{bmatrix}
1 & \frac{1}{x_0}\bar{x}^\top \\
\frac{1}{x_0}\bar{x} & I_{n-1}
\end{bmatrix}.
$$
This completes the proof.
\end{proof}

It follows from Lemma \ref{lem:socp} that $ \frac{1}{\sigma}(I -D) + \tau I = \frac{1}{\sigma}\left( \begin{bmatrix}
  \tilde{a}_0 & \\
  & \tilde{a}_1 I
\end{bmatrix} - auu^{\mathrm{T}} \right) :=  \frac{1}{\sigma}\left( \Lambda_1 - a uu^{\mathrm{T}} \right), \Lambda_1 = (1 +\sigma \tau )I - \Lambda.$ Consequently, we can obtain from the SMW formula that $(\frac{1}{\sigma}( I - D) + \tau I)^{-1} = \sigma (\Lambda_1^{-1} + c \Lambda_1^{-1} u u^{\mathrm{T}} \Lambda_1^{-1} ) $, where $c =  \frac{a}{1 - a{u}^{\mathrm{T}} \Lambda_1^{-1} {u}}$ is a constant. Hence, the following equality holds.
\[
\begin{aligned}
& \sigma D + D \left( \frac{1}{\sigma}(I- D) +\tau I \right)^{-1}D   \\
 = & \sigma \left( \Lambda + auu^{\mathrm{T}} \right) +  \sigma(\Lambda + a uu^{\mathrm{T}})(\Lambda_1^{-1} + c \Lambda_1^{-1} u u^{\mathrm{T}} \Lambda_1^{-1} )(\Lambda + a uu^{\mathrm{T}}) \\
 = & \sigma \left(\Lambda \Lambda_1^{-1} \Lambda + \Lambda + c\Lambda \Lambda_1^{-1}uu^{\mathrm{T}}  \Lambda_1^{-1}\Lambda +  a(1+c\gamma) \Lambda \Lambda_1^{-1}uu^{\mathrm{T}} + a(1 + c\gamma)uu^{\mathrm{T}} \Lambda_1^{-1}\Lambda + \left( a + a^2 \gamma + a^2 c \gamma^2\right) uu^{\mathrm{T}}  \right) \\
 = & \sigma \left(\tilde{\Lambda} + \begin{bmatrix}
     b_0 u_0^2 & b_1 u_0 u_1^{\mathrm{T}} \\
    b_1 u_0 u_1 & b_2 u_1 u_1^{\mathrm{T}}
 \end{bmatrix} \right),  \\
\end{aligned}
\]
where $\tilde{\Lambda} = \Lambda \Lambda_1^{-1} \Lambda + \Lambda , \gamma = u^{\mathrm{T}} \Lambda_1^{-1} u$, and $b_0,b_1,b_2$ are constants. Denote $\Lambda_1^{-1} \Lambda = \begin{bmatrix}
  c_0 & \\
  & c_1 I
\end{bmatrix}$, we have $b_0 = cc_0^2 + 2a(1+c\gamma)c_0 + a + a^2\gamma + a^2c\gamma^2, b_1 = cc_0c_1 + a(1+c\gamma)(c_0 + c_1) + a + a^2\gamma + a^2c\gamma^2, b_2 = cc_1^2 + 2a(1+c\gamma)c_1 + a + a^2\gamma + a^2c\gamma^2$. Consequently, let $\tilde{u} = [b_1/b_2 u_0; u_1] $, $ b_2\tilde{u} \tilde{u}^{\mathrm{T}} = \begin{bmatrix}
    b_1^2/b_2 u_0^2& b_1 u_0u_1^{\mathrm{T}} \\
    b_1 u_0u_1 & b_2 u_1 u_1^{\mathrm{T}}
\end{bmatrix}$, it follows that
\[
\overline{D} = \sigma D + D \left( \frac{1}{\sigma}(I- D) +\tau I \right)^{-1} D  =  \sigma \left(\tilde{\Lambda} + \begin{bmatrix}
    (b_0 - b_1^2/b_2) u_0^2 & 0 \\
   0  & {0}
\end{bmatrix} \right) + \sigma b_2 \tilde{u} \tilde{u}^{\mathrm{T}}.
\]
Hence, the linear system can be represented as a diagonal matrix plus a rank-one matrix, which is significant in constructing the Schur matrix when solving the linear system using direct methods such as Cholesky factorization.


 \subsection{Spectral functions}

 The spectral type functions include $p(\bm{X}) =  \lambda \|\bm{X}\|_*,  \lambda \|\bm{X}\|_2$ and $\delta_{\mathbb{S}_+^n}(\bm{X})$. For more details of the generalized Jacboian of spectral functions, we refer the readers to \cite{themelis2019acceleration}. We present a non-exhaustive introduction to the usually used spectral function in the following.  For a given $\bm{X} \in \mathbb{R}^{n_1 \times n_2}$, let the singular value decomposition of $\bm{X}$ denoted by $\bm{X} = U \Sigma V^{\mathrm{T}}$, then the proximal operator of $\lambda \|\bm{X}\|_*$ can be presented by:
\begin{equation} \label{prox:nuclear}
\text{prox}_{\lambda \|\cdot\|_*}(\bm{X}) = U \text{diag} \left( T_{\lambda} (\sigma(\bm{X})) \right) V^T,
\end{equation}
where $T_{\lambda}(\cdot)$ denotes the soft shrinkage operator. Without loss of generality, we consider the case that $n_2 \ge n_1$. Let $V = [V_1, V_2]$ with $V_1 \in \mathbb{R}^{n_1 \times n_1}$ and $V_2 \in \mathbb{R}^{n_1 \times (n_2 -n_1)}$, then one generalized Jacobian $D$ of \eqref{prox:nuclear} is
\begin{equation} \label{genJacobian-nuclear}
\begin{aligned}
    D(G) &= U \left[    \frac{\Omega_{\sigma,\sigma}^{\mu} + \Omega_{\sigma,-\sigma}^{\mu}}{2} \odot G_1 +   \frac{\Omega_{\sigma,\sigma}^{\mu} - \Omega_{\sigma,-\sigma}^{\mu}}{2} \odot G_1^{\top}, (\Omega_{\sigma,0}^{\mu} \odot\left( G_2 \right)) \right] V^{\top},
\end{aligned}
\end{equation}
where $\odot$ denotes the Hadamard product, $\sigma =[\sigma^{(1)};\cdots;\sigma^{(m)}] \in \mathbb{R}^m$  is the singular value of $\bm{X}$,  $ G_1 = U^{\top} G  V_1 \in \mathbb{R}^{n_1 \times n_1 }, G_2 = U^{\top} G V_2 \in \mathbb{R}^{n_1 \times  (n_2 - n_1)  }$ and $\Omega_{\sigma,\sigma}^{\lambda}$ is defined by:
\begin{equation}
    (\Omega_{\sigma,\sigma}^{\lambda})_{ij} := \begin{cases}
  \partial_B \prox_{\lambda \| \cdot \|_1}(\sigma_i ) , & \mbox{if } \sigma_i = \sigma_j, \\
  \left\{ \frac{\prox_{\lambda\|\cdot\|_1}(\sigma_i) -\prox_{\lambda\|\cdot\|_1}(\sigma_j) }{\sigma_i-\sigma_j} \right\}, & \mbox{otherwise}.
\end{cases}
\end{equation}

 For $p(\bm{X}) = \lambda \|\bm{X} \|_2$, its proximal operator can be represented by
\begin{equation} \label{prox:l2l2}
 \text{prox}_{\lambda \|\cdot\|_2}(\bm{X}) =    U\text{diag} \left( \lambda(\bm{X}) - \lambda P_{\Delta_n} (\lambda(\bm{X}) / \lambda) \right) V^T,
\end{equation}
where $P_{\Delta_n}$ denotes the projection onto simplex unit $\Delta_n := \{\bm{x} \in \mathbb{R}^n | \bm{1}^{\mathrm{T}}\bm{x} = 1, \bm{x} \ge 0 \}.$
Hence, the generalized Jacobian of \eqref{prox:l2l2} is  \eqref{genJacobian-nuclear} with
\begin{equation} \label{prox:l2l22}
    (\Omega_{\sigma,\sigma}^{\lambda})_{ij} := \begin{cases}
  \partial_B (\prox_{\lambda \| \cdot \|_{\infty}}(\sigma))_i , & \mbox{if } \sigma_i = \sigma_j, \\
  \left\{ \frac{(\prox_{\lambda\|\cdot\|_{\infty}}(\sigma))_i -(\prox_{\lambda\|\cdot\|_{\infty} }(\sigma))_j }{\sigma_i-\sigma_j} \right\}, & \mbox{otherwise}.
\end{cases}
\end{equation}
For $p(\bm{X}) = \delta_{\mathbb{S}_+^n}( \bm{X} )$, the corresponding generalized Jacobian operator can also be written as
\begin{equation} \label{genJacobian-sdp}
 D_{\mathbb{S}_+^n }(H):=Q\left(\Sigma \odot\left(Q^{\mathrm{T}} H Q\right)\right) Q^{\mathrm{T}}, \quad H \in \mathbb{S}^n,
\end{equation}
where
\begin{equation}\label{eq:sdp+:sigma}
    \Sigma=\left[\begin{array}{cc}
E_{\alpha \alpha} & v_{\alpha \bar{\alpha}} \\
v_{\alpha \bar{\alpha}}^{\mathrm{T}} & 0
\end{array}\right], \quad v_{i j}:=\frac{\lambda_i}{\lambda_i-\lambda_j}, \quad i \in \alpha, \quad j \in \bar{\alpha},
\end{equation}
where $E_{\alpha \alpha} \in \mathbb{S}^{|\alpha|}$ is the matrix of ones.

For given $\sigma$, denote $D^{\tau} = \frac{1}{\sigma}(I -D) + \tau I$, we next introduce a lemma that will be used to obtain $(D^{\tau })^{-1}$ and preserve the low-rank structure. Since it can be verified directly,  we omit the proof.

\begin{lemma}
Let $\mathcal{T}: \mathbb{R}^{n \times n} \rightarrow \mathbb{R}^{n \times n} : \mathcal{T}(G) = \Omega_1 \odot G + \Omega_2 \odot G^{\mathrm{T}}$, then the inverse of  $\mathcal{T}$ is
\[
\mathcal{T}^{-1}(G) = (\Omega_s + \Omega_a) \odot G + (\Omega_s - \Omega_a) \odot G^{\mathrm{T}},
\]
where $\Omega_s = 1./[2(\Omega_1 + \Omega_2)],\, \Omega_a = 1./[2(\Omega_1 - \Omega_2)]$ and $./$ denotes elementwise division.
\end{lemma}

According to the above lemma, $(D^{\tau})^{-1}$ can be represented as
\begin{equation}
\begin{aligned}
(D^{\tau})^{-1}(G) &= U \ \left[ (\Omega^{\tau}_s + \Omega^{\tau}_a) \odot G_1 +    (\Omega^{\tau}_s - \Omega^{\tau}_a) \odot G_1^{\mathrm{T}} , (1./\Omega^{\tau}_3  \odot G_2 ) \right] V^{\mathrm{T}} + \sigma/(1+\sigma \tau) G.
\end{aligned}
\end{equation}
where $\Omega^{\tau}_s = 1./[2(\Omega_1^{\tau} + \Omega_2^{\tau} ) ] - \sigma/(1+\sigma \tau)E ,\, \Omega^{\tau}_a = 1./[2(\Omega_1^{\tau} -\Omega_2^{\tau} ) ],\, E$ is the matrix of ones with the correct size. The details of the computational process are summarized in Algorithm \ref{alg:computeD2}, from which the low-rank structure can be exploited effectively, and the total computational cost for each inner iteration reduces to $O(nr^2)$.

  \begin{algorithm}[h]
\caption{The process of computing $(D^{\tau})^{-1}(G)$.}
\begin{algorithmic}[1] \label{alg:computeD2}
\Require $G,(\Omega_1^{\tau})_{\alpha \alpha},(\Omega_1^{\tau})_{\alpha \bar{\alpha}}, (\Omega_2^{\tau})_{\alpha \alpha},\,(\Omega_2^{\tau})_{\alpha \bar{\alpha}},\,$ $(\Omega_1^{\tau})_{\alpha \beta}, U=[U_{\alpha},U_{\bar{\alpha}}],V = [V_{\alpha},\,V_{\bar{\alpha}},\, V_{\beta}]$, where $U_{\alpha} \in \mathbb{R}^{n_1 \times r}, U_{\bar{\alpha}} \in \mathbb{R}^{n \times (n_1-r)}, V_{\alpha} \in \mathbb{R}^{n_2 \times r}, V_{\bar{\alpha}} \in \mathbb{R}^{n_2 \times (n_2-r)},
V_{\beta} \in  \mathbb{R}^{n_2 \times (n_2 -n_1)}$, penalty parameter $\sigma$ and regularizer parameter $\tau$.
\Ensure $(D^{\tau})^{-1}(G)$
\State Compute $(G_1)_{\alpha\alpha}, (G_1)_{\alpha\bar{\alpha}}, (G_1)_{\bar{\alpha}\alpha}, (G_1)_{\alpha\beta}$  where
\[
\begin{aligned}
(G_1)_{\alpha\alpha} &= U_{\alpha}^{\mathrm{T}}\,\, G\,\, (V_{\alpha}),\qquad(G_1)_{\alpha\bar{\alpha}}= U_{\alpha}^{\mathrm{T}} \,\,G \,\,(V_{\bar{\alpha}}), \\
(G_1)_{\bar{\alpha}\alpha} &= U_{\bar{\alpha}}^{\mathrm{T}} \,\,G \,\,(V_{\bar{\alpha}}), \qquad
(G_1)_{\alpha\beta} = U_{\alpha}^{\mathrm{T}}\,\, G\,\, (V_{\beta}).
\end{aligned}
\]
\State Compute
\[
 G_2 = \begin{bmatrix}
(G_2)_{\alpha \alpha} & (G_2)_{\alpha \bar{\alpha}} & (G_2)_{\alpha \beta}\\
(G_2)_{\bar{\alpha} \alpha} & 0 &0
 \end{bmatrix},
\]
where
\[
\begin{aligned}
(G_2)_{\alpha \alpha} &= (\Omega_1^{\tau})_{\alpha \alpha} \odot (G_1)_{\alpha \alpha} + (\Omega_2^{\tau})_{\alpha \alpha} \odot ((G_1)_{\alpha \alpha})^{\mathrm{T}} ,\\
(G_2)_{\alpha \bar{\alpha}} &= (\Omega_1^{\tau})_{\alpha \bar{\alpha}} \odot (G_1)_{\alpha \bar{\alpha}} + (\Omega_2^{\tau})_{\alpha \bar{\alpha}} \odot ((G_1)_{\bar{\alpha} \alpha})^{\mathrm{T}},\\
(G_2)_{\bar{\alpha} \alpha} &= ( (\Omega_1^{\tau})_{\alpha \bar{\alpha}})^{\mathrm{T}} \odot (G_1)_{\bar{\alpha} \alpha} + ((\Omega_2^{\tau})_{\alpha \bar{\alpha}})^{\mathrm{T}} \odot ((G_1)_{\alpha \bar{\alpha}})^{\mathrm{T}}, \\
(G_2)_{\alpha \beta} &= (\Omega_1^{\tau})_{\alpha \beta} \odot (G_1)_{\alpha \beta}.
\end{aligned}
\]
\State Compute $G_3 = G_{12} + G_{11} + G_{21} + G_{13}$ where
\[
\begin{aligned}
G_{11} & = U_{\alpha}\,\, (G_2)_{\alpha \alpha} \,\,V_{\alpha}^{\mathrm{T}}, \qquad
G_{12}  = U_{\alpha}\,\, (G_2)_{\alpha \bar{\alpha}} \,\,V_{\bar{\alpha}}^{\mathrm{T}}, \\
G_{21} & = U_{\bar{\alpha}}\,\, (G_2)_{\bar{\alpha} \alpha} \,\,V_{\alpha}^{\mathrm{T}}, \qquad
G_{13}  = U_{\alpha}\,\, (G_2)_{\alpha \alpha} \,\,V_{\beta}^{\mathrm{T}}.
\end{aligned}
\]
\State Compute $(D^{\tau})^{-1}(G) = G_3 + \sigma/(1+\sigma \tau)G$.
\end{algorithmic}
\end{algorithm}

\subsection{Fused regularizer}
For the fused regularizer $p(x) = \lambda_1 \|x\|_1 + \lambda_2 \|Fx\|_1,$ where $F(x) = [x_2 - x_1,\cdots,x_n - x_{n-1}]$,  it follows from \cite[Proposition 4]{li2018efficiently} that the proximal operator of $p$ is
\begin{equation} \label{prox-fused}
\text{prox}_p(\bm{v}) = \text{prox}_{\lambda_1 \|\cdot \|_1}(x_{\lambda_2}(\bm{v})) = \text{prox}_{\lambda_1 \|\cdot \|_1}( \bm{v} - F^{\mathrm{T}}z_{\lambda_2}(F\bm{v})  ),
\end{equation}
 where
$
z_{\lambda_2}(u) := \operatorname{argmin}_z \left\{ \frac{1}{2} \| F^T z \|^2 - \langle z, u \rangle \ \bigg| \ \| z \|_\infty \leq \lambda_2 \right\} , \forall u \in \mathbb{R}^{n-1}.
$ To characterize the generalized Jacobian of \eqref{prox-fused}. we define the  multifunction $\mathcal{M}: \mathbb{R}^n \rightarrow \mathbb{R}^{n \times n}$ as:
\begin{equation} \label{genJacobian:fused}
\mathcal{M}(v):= \{M \in \mathbb{R}^{n \times n}| M = \Theta P, \Theta \in \partial_{B}\prox_{\lambda_1 \|\cdot\|_1}(x_{\lambda_2}), P \in \mathcal{P}_x(v) \},
\end{equation}
where
$
\mathcal{P}_x(v):= \{\hat{P} \in \mathbb{R}^{n-1 \times n-1}| \hat{P} = I - F^{\mathrm{T}}(\Sigma_K FF^{\mathrm{T}}\Sigma_K )^{\dagger}F, K \in \mathcal{K}_z(v) \},
$ $\Sigma_K = \text{Diag}(\sigma_K) \in \mathbb{R}^{(n-1) \times (n-1)}$ and
\[
(\sigma_K)_i = \begin{cases}
                 0, & \mbox{if } i \in K, \\
                 1, & \mbox{otherwise}, i = 1,\cdots,n-1.
               \end{cases}
\]
It follows from \cite[Theorem 2]{li2018efficiently} that $\mathcal{M}$ is nonempty and can be regarded as the generalized Jacobian of
$\text{prox}_p$ at $\bm{v}$. Furthermore, any element in $\mathcal{M}$ is symmetric and positive semidefinite. Let $\Gamma:= I_n - F^{\mathrm{T}}(\Sigma FF^{\mathrm{T}}\Sigma)^{\dagger}F  = \text{Diag}(\Gamma_1,\cdots,\Gamma_N),$ where
\[
\Gamma_i = \begin{cases}
             \frac{1}{n_i + 1}\mathbf{E}_{n_i+1}, & \mbox{if } i \in J, \\
             I_{n_i}, & \mbox{if } i \in \{1,N\}, \\
             I_{n_i-1}, & \mbox{otherwise}.
           \end{cases}
\]
It follows that $\Gamma = H + UU^{\mathrm{T}} = H + U_JU_J^{\mathrm{T}},$ where $H \in \mathbb{R}^{n \times n}$ is an N-block diagonal matrix given by $H = \text{Diag}(\Upsilon_1,\dots,\Upsilon_N)$ with
\[
\Upsilon_i =
\begin{cases}
O_{n_i+1}, & \text{if } i\in J,\\
I_{n_i}, & \text{if } i\notin J\text{ and } i\in\{1,N\},\\
I_{n_i-1}, & \text{otherwise}.
\end{cases}
\]
Furthermore, the $(k,j)$-th entry of $ U \in \mathbb{R}^{n\times N} $ is  given by
\begin{equation}
U_{k,j}=
\begin{cases}
\dfrac{1}{\sqrt{n_j+1}}, & \text{if }\displaystyle\sum_{t=1}^{j-1} n_t + 1 \leq k \leq \sum_{t=1}^{j} n_t + 1,\quad\text{and }j\in J,\\[10pt]
0, & \text{otherwise,}
\end{cases}
\end{equation}
and  $U_J$ consists of the nonzero columns of $U$, i.e., the columns indexed by $J$. Then $D = \Theta P \in \mathcal{M},$ where $P = I - F^{\mathrm{T}}(\Sigma FF^{\mathrm{T}}\Sigma)^{\dagger}F,$ and
\[
\theta_i = \begin{cases}
             0, & \mbox{if } |(x_{\lambda_2}(v))_i |\le \lambda_1,  \\
             1, & \mbox{otherwise, \quad $i = 1,\cdots ,n $}.
           \end{cases}
\]
 Let $I_z(v) :=  \{i| |(z_{\lambda_2}(Bv))_i|=\lambda_2, i =1,\cdots,n-1 \}$, then $\Sigma = \text{Diag}(\sigma) \in \mathbb{R}^{(n-1) \times (n-1)}$ with
\[
\sigma_i = \begin{cases}
             0, & \mbox{if } i \in I_z(v), \\
             1 , & \mbox{otherwise}, i =1,\cdots,n-1.
           \end{cases}
\]
It follows that $\Theta \in \partial_{B} \prox_{\lambda_1 \|\cdot\|_1}(x_{\lambda_2}(v))$ and $P \in \mathcal{P}_x(v).$
Therefore, we have $M = \Theta (H + U_JU_J^{\mathrm{T}}) = \Theta (H + U_JU_J^{\mathrm{T}})\Theta, \, \Theta^2 = \Theta,\, H^2=H,\, \Theta H = \Theta H \Theta.$
Define the index sets
$
\alpha_1:= \{i| \theta_i = 1, i \in \{1,\cdots,n\} \},\quad \alpha_2:= \{i\,|h_i = 1, i \in \alpha_1 \},
$
where $\theta_i$ and $h_i$ are the $i$-th diagonal entries of matrices $\Theta$ and $H$ respectively. It then follows that
\[
\mathcal{B} \Theta H \mathcal{B}^{\mathrm{T}} = \mathcal{B} \Theta H \Theta \mathcal{B}^{\mathrm{T}} = \mathcal{B}_{\alpha_1} H \mathcal{B}_{\alpha_1}^{\mathrm{T}} = \mathcal{B}_{\alpha_2} \mathcal{B}_{\alpha_2}^{\mathrm{T}},
\]
where $\mathcal{B}_{\alpha_1} \in \mathbb{R}^{m \times |\alpha_1|}$ and $\mathcal{B}_{\alpha_2} \in \mathbb{R}^{m \times |\alpha_2|}$ are two submatrices obtained from $\mathcal{B}$ by extracting those columns with indices in $\alpha_1$ and $\alpha_2.$ Meanwhile, we hvae
\[
\mathcal{B}\Theta(U_j U_j^{\mathrm{T}})\mathcal{B}^{\mathrm{T}} = \mathcal{B}\Theta(U_j U_j^{\mathrm{T}}) \Theta \mathcal{B}^{\mathrm{T}} =
\mathcal{B}_{\alpha_1} \tilde{U} \tilde{U}^{\mathrm{T}} \mathcal{B}_{\alpha_1}^{\mathrm{T}},
\]
where $\tilde{U} \in \mathbb{R}^{|\alpha_1| \times r}$ is a submatrix obtained from $\Theta U_J$ by extracting those rows with
indices in $\alpha_1$ and the zero columns in $\Theta U_J$ are removed. Therefore, by exploiting the
structure in $D$, $\mathcal{B}D\mathcal{B}^{\mathrm{T}}$ can be expressed in the following form:
\[
\mathcal{B}D\mathcal{B}^{\mathrm{T}} = \mathcal{B}_{\alpha_2} \mathcal{B}_{\alpha_2}^{\mathrm{T}} + \mathcal{B}_{\alpha_1}\tilde{U}\tilde{U}^{\mathrm{T}}\mathcal{B}_{\alpha_1}^{\mathrm{T}}.
\]
For given $\sigma$, $D(D^{\tau})^{-1}D$, where $D^{\tau} = \frac{1}{\sigma}(I- D) + \tau I$ and $D = \Theta H \Theta,$  we note that $D = \Theta H \Theta = \Theta H$ holds since
$
\Theta = \text{Diag}(\Theta_1,\cdots,\Theta_N).
$
It yields that $M = \text{Diag}(\Theta_1 \Gamma_1,\cdots,\Theta_N \Gamma_N).$ Define $J := \{j| \, \Gamma_j \, \mbox{is not an identity matrix}, 1\le j \le N \}.$ It follows from $\text{supp}(Fx_{\lambda_2}(v)) \subset K$ that
$
\Theta_j = \mathbf{O}_{n_j +1}\, \mbox{or}\, I_{n_j +1 }, \forall j \in J,
$
which implies $\Theta_j \Gamma_j \in \mathbb{S}_+^{n_j + 1}, \forall j \in J$ and hence $D \in \mathbb{S}_+^n$. Consequently, we have
$D = \text{Diag}(D_1,\cdots,D_n),$
where
\[
D_i = \begin{cases}
        \frac{1}{n_i +1}\mathbf{E}_{n_i +1}, & \mbox{if } i \in J \,\mbox{and}\, \Theta_i =I_{n_i}, \\
        I_{n_i}, & \mbox{if }\, i \notin J\, \mbox{and}\, i \in \{i,N\}, \\
        0, & \mbox{if } \Theta_i = \bm{0}, \\
        I_{n_i -1}, & \mbox{otherwise}.
      \end{cases}
\]
According to the SMW formula,
the inverse of $(D^{\tau})^{-1}$ has the explicit solution:
\[
\begin{aligned}
&\left((\frac{1}{\sigma} + \tau) I - \frac{1}{\sigma}D \right)^{-1} = \left((\frac{1}{\sigma} + \tau) I - \frac{1}{\sigma}\Theta(H + U_JU_J^{\mathrm{T}}) \Theta \right)^{-1}  \\
&= \begin{cases}
                 \frac{\sigma}{1+\tau \sigma} I_{n_i +1} + \frac{1}{\tau(1 +\tau\sigma) (n_i +1)}\mathbf{E}_{n_1 +1} , & \mbox{if } \Theta_i = 1, \\
                 \frac{1}{\tau} I_{n_i} , & \mbox{if } \mbox{if }\, i \notin J\, \mbox{and}\, i \in \{i,N\}, \\
                 \frac{\sigma}{1+\sigma \tau} I_{n_i}, & \mbox{if } \Theta_i = \bm{0}, \\
                 \frac{1}{\tau} I_{n_i -1} , & \mbox{otherwise}.
               \end{cases}
\end{aligned}
\]
Consequently, $\overline{D} = \sigma D + D (D^\tau)^{-1} D$ can be represented by:
\[
\begin{aligned}
\overline{D}&= \begin{cases}
                 (\frac{1}{\tau} + \sigma ) \frac{1}{n_i +1 }\mathbf{E}_{n_1 +1}  , & \mbox{if } i \in J \,\mbox{and}\, \Theta_i =I_{n_i}, \\
                 (\frac{1}{\tau}+\sigma)  I_{n_i} , & \mbox{if } \mbox{if }\, i \notin J\, \mbox{and}\, i \in \{i,N\}, \mbox{and}\, \Theta_i =I_{n_i}, \\
                 0 , & \mbox{if } \Theta_i = \bm{0}, \\
                 (\frac{1}{\tau} +\sigma) I_{n_i -1} , & \mbox{otherwise},
               \end{cases}
\end{aligned}
\]
and $D(D^{\tau})^{-1}$ is:
\[
D (D^{\tau})^{-1} = \begin{cases}
                 \frac{1}{\tau (n_i +1 )} \mathbf{E}_{n_1 +1}  , & \mbox{if } i \in J \,\mbox{and}\, \Theta_i =I_{n_i}, \\
                 \frac{1}{\tau} I_{n_i} , & \mbox{if } \mbox{if }\, i \notin J\,,  i \in \{i,N\}, \mbox{and}\, \Theta_i =I_{n_i}, \\
                 0, & \mbox{if } \Theta_i = \bm{0}, \\
                 \frac{1}{\tau} I_{n_i -1} , & \mbox{otherwise}.
               \end{cases}
\]
Note that $\overline{D} = \Theta \tilde{H}$ and hence we have
$
\mathcal{B} \Theta \tilde{H} \mathcal{B}^{\mathrm{T}} = \mathcal{B} \Theta \tilde{H} \Theta \mathcal{B}^{\mathrm{T}} = \mathcal{B}_{\alpha_1} \widetilde{U}\widetilde{U} \mathcal{B}_{\alpha_1}^{\mathrm{T}} + \widetilde{\mathcal{B}}_{\alpha_2} \widetilde{\mathcal{B}}_{\alpha_2}^{\mathrm{T}},
$ where $\widetilde{U}$ and $\widetilde{\mathcal{B}}_{\alpha_2}$ are the scaling matrices of $U$ and $\mathcal{B}_{\alpha_2}$. This yields the decomposition: $\mathcal{B} \overline{D} \mathcal{B}^{\mathrm{T}} = W_1 W_2^{\mathrm{T}},$
where $W_1 := [\widetilde{\mathcal{B}}_{\alpha_2}, \mathcal{B}_{\alpha_1}\widetilde{U}\widetilde{U}^{\mathrm{T}} ] \in \mathbb{R}^{m \times (|\alpha_1| + |\alpha_2|)}, W_2 = [\widetilde{\mathcal{B}}_{\alpha_2}, \mathcal{B}_{\alpha_1}].$
Using the above decomposition, we obtain
\[
((\tau_1 +1)I + \mathcal{B}\overline{D}\mathcal{B}^{\mathrm{T}})^{-1} = \frac{1}{\tau_1 + 1} I_m - \frac{1}{\tau_1 + 1} W_1 ((\tau_1 + 1)I_{|\alpha_1| + |\alpha_2|}+ W_2^{\mathrm{T}}W_1)^{-1}W_2^{\mathrm{T}}.
\]
Hence, we only need to factorize an $(|\alpha_1| + |\alpha_2|) \times (|\alpha_1| + |\alpha_2|)$ matrix and the total computational cost is  merely $\mathcal{O}(|\alpha_1| + |\alpha_2|)^3 + \mathcal{O}(m(|\alpha_1|+|\alpha_2|)^2)$, matching the result  in \cite{li2018efficiently}. Consequently, we can solve the linear system using direct methods such as Cholesky factorization at low cost.
\section{Numerical experiments} \label{4}
In this section, we conduct numerous experiments on different kinds of problems to verify the efficiency and robustness of Algorithm \ref{alg:ssn}.   The criteria to measure the accuracy are based on the KKT optimality conditions:
\[
\eta = \max\{\eta_P,\eta_D,\eta_K,\eta_{\mathcal{P}} \},
\]
where
\begin{equation*}\label{criteria}
\begin{aligned}
\eta_P &:= \frac{\| \mathcal{A}(\bm{x}) - \Pi_{\mathcal{P}_2}(\mathcal{A}(\bm{x}) - \bm{y} ) \|  }{1+\|
\bm{x}\|} ,
\eta_D := \frac{\| \mathcal{A}^*(\bm{y}) + \mathcal{B}^*(\bm{z}) + \bm{s} - \mathcal{Q}(\bm{v})- \bm{c}\|  }{1 + \|\bm{c}\|}, \\
\eta_K &:= \min \left\{ \frac{\| \bm{x}-\text{prox}_{p}(\bm{x}-\bm{s} )\| }{1+\|\bm{s}\|+\| \bm{x}\| },  \frac{ \|\mathcal{Q}(\bm{v}) -\mathcal{Q}(\bm{x})\|_{\mathrm{F}} }{1 + \|\mathcal{Q}(\bm{v}) \| + \|\mathcal{Q}(\bm{x}) \|  } \right\}, \\
\eta_{\mathcal{P}} &:= \min \left\{
   \frac{\|\text{prox}_{ f}(\mathcal{B}\bm{x} -  \bm{z}) -  \mathcal{B}(\bm{x})\|}{1 + \|\mathcal{B}(\bm{x})\| + \|\bm{z}\| } \,\text{or}\, \frac{\| -\nabla f^*(-\bm{z}) - \mathcal{B}(\bm{x})\| }{1+ \|\mathcal{B}(\bm{x})\| + \|\bm{z} \| },
\frac{\|\Pi_{\mathcal{P}_1}(\bm{x} -  \bm{r} ) -  \bm{x}\|}{1+ \|\bm{x}\| + \|\bm{r}\| } \right\}.
\end{aligned}
\end{equation*}
Denote \texttt{pobj} and \texttt{dobj} as the primal and dual objective function value. We also compute the relative gap by
 \[
 \eta_g = \frac{\texttt{|pobj - dobj|} }{1 + \texttt{|pobj| + |dobj|}}.
 \]
Our software is available at \url{https://github.com/optsuite/SSNCVX}. All the experiments are done on a Linux server with a sixteen-core Intel Xeon Gold 6326 CPU and 256G memory.


\subsection{Lasso}
The Lasso problem corresponding to \eqref{general} can be expressed as
\begin{equation}\label{lasso}
\min_{\bm{x}} \quad  \frac{1}{2}\|\mathcal{B}(\bm{x})-\bm{b}\|^2 + \lambda \|\bm{x}\|_1.
\end{equation}
We test the problem on data from UCI\footnote{\url{ https://archive.ics.uci.edu/}}
   and LIBSVM dataset\footnote{\url{ https://www.csie.ntu.edu.tw/~cjlin/libsvmtools/datasets/}}.
These datasets are collected
from the 10-K Corpus \cite{kogan2009predicting} and the UCI data repository \cite{lichman2013uci}.
As suggested in \cite{huang2010predicting}, for the datasets \textbf{pyrim}, \textbf{triazines}, \textbf{abalone}, \textbf{bodyfat}, \textbf{housing}, \textbf{mpg}, and
\textbf{space\_ga}, we expand their original features by using polynomial basis functions over
those features \cite{li2018highly}. For example, the last digit in \textbf{pyrim5} indicates that an order 5 polynomial is used to generate the basis functions. This naming convention is also used
in the rest of the expanded data sets. These numerical instances, shown in Table \ref{tab:tensor-image}, can be
quite difficult in terms of the dimensions and the largest eigenvalue of $\mathcal{B}\mathcal{B}^*$,
which is denoted as $\lambda_{\max}(\mathcal{B}\mathcal{B}^*)$.

 In Table \ref{tab:1e3}, $m$ denotes the number of samples, $n$ denotes the number of features, and ``nnz'' denotes the number of nonzeros in the solution $x$  using the following estimation:
      \[
      \text{nnz}:= \min \{k| \sum_{i=1}^{k} |\hat{x}_i| \ge 0.999 \|x\|_1 \},
      \]
      where $\hat{x}$ is obtained by sorting $x$ such that $|\hat{x}_1| \ge |\hat{x}_2|  \ge \cdots \ge |\hat{x}_n|.$
      The algorithms to compare are SSNAL\footnote{\url{ https://github.com/MatOpt/SuiteLasso}}, SLEP \cite{liu2009slep}, and the ADMM algorithm. The numerical results for different choice of $\lambda$, i.e., $\lambda = 10^{-3}\|\mathcal{B}^{\mathrm{T}}\bm{b}\|_{\infty}$ and $\lambda = 10^{-4}\|\mathcal{B}^{\mathrm{T}}\bm{b}\|_{\infty}$ and different algorithms are given in Tables \ref{tab:1e3} and \ref{tab:1e4}, where "nnz" denotes the number of nonzeros in the solution. We can see that both SSNCVX and SSNAL have successfully solved all problems, while other first-order methods can not. Furthermore, SSNCVX is competitive with SSNAL in all the tested Lasso problems, demonstrating its superiority in solving Lasso problems. For example, for the instance \textbf{log1p.E2006.train}, SSNCVX is twice as fast as SSNAL, while under the maximum time limit, SLEP and ADMM only achieve accuracies of 2.0e-2 and 1.2e-1, respectively.

\begin{table}[H]
\centering

\begin{tabular}{|c|c|c|}
\cline{1-3}
Probname      &  $(m,n)$ & $\lambda_{\max}(\mathcal{B}\mathcal{B}^*)$  \\ \cline{1-3}
E2006.train & (3308, 72812) & 1.912e+05 \\ \cline{1-3}
log1p.E2006.train  & (16087,4265669) & 5.86e+07  \\ \cline{1-3}
 E2006.test & (3308,72812) &  4.79e+04 \\ \cline{1-3}
 log1p.E2006.test & (3308,1771946) & 1.46e+07 \\ \cline{1-3}
 pyrim5 &  (74,169911) & 1.22e+06 \\ \cline{1-3}
 triazines4 & (186,557845) & 2.07e+07 \\ \cline{1-3}
abalone7 &  (4177,6435) & 5.21e+05 \\ \cline{1-3}
bodyfat7 & (252,116280) & 5.29e+04 \\ \cline{1-3}
housing7 & (506,77520) & 3.28e+05 \\ \cline{1-3}
mpg7 & (392,3432) & 1.28e+04 \\ \cline{1-3}
spacega9 & (3107,5005) & 4.01e+03  \\ \cline{1-3}
  \end{tabular}
  \caption{
Statistics of the UCI test instances.}\label{tab:tensor-image}
\end{table}

\begin{table}[h]
\begin{tabular}{|c|c|cc|cc|cc|cc|}
\cline{1-10}
\multirow{2}{*}{id}&\multirow{2}{*}{nnz} &  \multicolumn{2}{c|}{SSNCVX} & \multicolumn{2}{c|}{SSNAL} & \multicolumn{2}{c|}{SLEP} & \multicolumn{2}{c|}{ADMM} \\
\cline{3-10}
&  &    $\eta$     & time        &  $\eta$  & time & $\eta$   & time & $\eta$  & time\\
\cline{1-10}
uci\_CT & 13  & 7.6e-7  & \textbf{0.64}    & 4.4e-13 & 0.86 &  2.2e-2  & 35.95 &  7.7e-3 & 46.02\\
\cline{1-10}
log1p.E2006.train & 5  & 5.4e-7  & \textbf{17.3}    & 1.5e-11 & 36.0 & 2.0e-2  & 1850.15 & 1.2e-1 & 3604.34\\
\cline{1-10}
E2006.test  & 1  & 2.2e-11  & \textbf{0.17}    & 4.3e-10 & 0.28 &  7.5e-12   & 1.11 & 7.9e-7 & 428.64 \\
\cline{1-10}
log1p.E2006.test   & 8 & 3.3e-8  & \textbf{2.83}    & 2.5e-10 & 5.12 & 4.8e-2  & 447.56  & 1.2e-1  & 3603.64 \\
\cline{1-10}
pyrim5  & 72  & 4.2e-16 & \textbf{1.82}    & 5.7e-8 & 2.16 & 2.4e-2  & 106.09  & 1.5e-3  & 3600.52 \\
\cline{1-10}
triazines4 & 519  & 2.6e-13  & \textbf{10.64}    & 3.4e-9 & 11.23 &  8.3e-2 & 246.11  & 9.7e-3 & 3603.99 \\
\cline{1-10}
abalone7 & 24  & 4.6e-11 & \textbf{0.75}    & 1.8e-9 & 1.06 & 2.5e-3 & 34.57 & 3.7e-4  & 540.27 \\
\cline{1-10}
bodyfat7 & 2 & 4.8e-13  & \textbf{0.79}    & 1.4e-8 & 1.08 & 1.9e-6 & 28.10 & 8.4e-4 & 3609.63 \\
\cline{1-10}
housing7 & 158  & 5.1e-13  & \textbf{1.83}    & 6.3e-9 & 1.74 &  1.3e-2  & 46.60 & 1.1e-2 & 3601.26 \\
\cline{1-10}
mpg7& 47 & 4.4e-16  & \textbf{0.10}   & 1.5e-8  & 0.14 & 7.4e-5   & 0.69 & 1.0e-6 & 63.41 \\
\cline{1-10}
spacega9 & 14  & 4.7e-15  & \textbf{0.25}   & 9.7e-9 & 1.01 & 1.9e-8   & 21.12 & 1.0e-6 & 294.52 \\
\cline{1-10}
E2006.train & 1   & 3.9e-9  & \textbf{0.44}    & 4.4e-10 & 0.87 &  1.4e-11   & 1.13 & 4.4e-5 & 1149.22\\
\cline{1-10}
\end{tabular}
\caption{The results on Lasso problem ($\lambda = 10^{-3}\|\mathcal{B}^{\mathrm{T}}\bm{b}\|_{\infty}$).}\label{tab:1e3}
\end{table}

\begin{table}[h]
\begin{tabular}{|c|c|cc|cc|cc|cc|}
\cline{1-10}
\multirow{2}{*}{id} & \multirow{2}{*}{nnz} & \multicolumn{2}{c|}{SSNCVX} & \multicolumn{2}{c|}{SSNAL} & \multicolumn{2}{c|}{SLEP} & \multicolumn{2}{c|}{ADMM} \\
\cline{3-10}  &  &  $\eta$     & time        &  $\eta$  & time & $\eta$  & time & $\eta$  & time\\
\cline{1-10}
uci\_CT  & 44 & 2.6e-7  & \textbf{1.26}    & 2.9e-12 & 1.75 & 1.8e-1 & 41.63 & 2.0e-3 & 49.88\\
\cline{1-10}
log1p.E2006.train & 599   & 3.0e-7  & \textbf{33.92}    & 5.9e-11 & 68.83 &  3.3e-2 & 1835.32 & 1.2e-1 & 3608.17 \\
\cline{1-10}
E2006.test& 1   & 2.6e-14  & \textbf{0.20}    & 3.7e-9 & 0.29 &  2.4e-12 & 0.38  & 9.0e-7 & 268.11\\
\cline{1-10}
log1p.E2006.test & 1081   & 8.8e-9  & \textbf{13.72}    & 2.7e-10 & 30.1 & 7.5e-2  & 455.56 & 1.6e-1 & 3606.60 \\
\cline{1-10}
pyrim5  & 78 & 5.6e-16  & \textbf{2.01}    & 5.0e-7 & 2.59 & 1.1e-2  & 108.93  & 3.1e-3 & 3601.09 \\
\cline{1-10}
triazines4 & 260  & 9.5e-16  & \textbf{18.48}    & 8.3e-8 & 34.44 & 9.2e-2 & 187.45  & 1.2e-2 & 3604.48 \\
\cline{1-10}
abalone7 & 59  & 6.1e-12  & \textbf{1.63}    & 1.2e-8 & 2.00 & 1.5e-2 & 43.91 & 1.0e-6 & 356.34 \\
\cline{1-10}
bodyfat7 & 3 & 1.0e-16 & \textbf{1.14}    & 9.7e-8 & 1.51 & 6.1e-4 & 41.98  & 1.3e-4 & 3601.89 \\
\cline{1-10}
housing7 & 281  & 2.6e-11  & \textbf{2.51}    & 1.2e-7 & 2.52 & 4.1e-2 & 52.60 & 3.6e-4 & 3601.09 \\
\cline{1-10}
mpg7 & 128 & 1.8e-15  & \textbf{0.11}   & 6.9e-8  & 0.18 & 5.8e-4  & 0.76 & 9.9e-7 & 11.67 \\
\cline{1-10}
spacega9  & 38 & 3.1e-12  & \textbf{0.53}   & 3.5e-7 & 0.72 & 9.0e-5  & 22.96 & 1.0e-6 & 53.23\\
\cline{1-10}
E2006.train & 1  & 5.6e-9  & \textbf{0.75}    & 4.4e-9 & 0.88 & 1.0e-11 & 1.39 & 4.4e-5 & 1132.34\\
\cline{1-10}
\end{tabular}
\caption{The results of tested algorithms on Lasso problem ($\lambda = 10^{-4}\|\mathcal{B}^{\mathrm{T}}\bm{b}\|_{\infty}$).}\label{tab:1e4}
\end{table}

\subsection{Fused Lasso}
The Fused Lasso problem corresponding to \eqref{general} can be expressed as
\begin{equation}\label{fused-lasso}
\min_{\bm{x}} \quad  \frac{1}{2}\|\mathcal{B}(\bm{x})-\bm{b}\|^2 + \lambda_1 \|\bm{x}\|_1 + \lambda_2 \| F \bm{x} \|.
\end{equation}
We compare SSNCVX with SSNAL \cite{li2018efficiently}, ADMM, and SLEP \cite{liu2009slep} solvers. Consistent with the Lasso problem, we also test the problems with data from the UCI data and the LIBSVM dataset.
The numerical experiments for UCI datasets are listed in Table \ref{tab:fus1e3}. It is shown that SSNCVX has comparable performance to SSNAL and better performance than ADMM and SLEP.

\begin{table}[h]
\begin{tabular}{|c|c|c|cc|cc|cc|cc|}
\cline{1-11}
\multirow{2}{*}{id} & \multirow{2}{*}{nnz($x$)}& \multirow{2}{*}{nnz($Bx$)} & \multicolumn{2}{c|}{SSNCVX} & \multicolumn{2}{c|}{SSNAL} & \multicolumn{2}{c|}{SLEP} & \multicolumn{2}{c|}{ADMM} \\ \cline{4-11}
 &  & & $\eta$     & time        & $\eta$  & time & $\eta$  & time & $\eta$  & time\\ \cline{1-11}
uci\_CT  & 8 & 1  & 6.3e-7  & \textbf{0.25}    & 7.9e-7 & 0.42 & 1.8e-6 & 2.06 & 7.7e-3 & 41.75 \\ \cline{1-11}
log1p.E2006.train  & 31 & 2  & 2.8e-7  & \textbf{10.43}    & 2.4e-7 & 14.02 & 1.2e-2 & 4889.15 & 1.2e-1 & 3623.18 \\ \cline{1-11}
E2006.test& 1 & 1 & 1.5e-7  & \textbf{0.17}    & 5.1e-7 & 0.33 & 4.8e-8 & 0.93 & 8.2e-7 & 1768.26\\ \cline{1-11}
log1p.E2006.test & 33 & 1 & 4.1e-7  & \textbf{2.60}    & 8.1e-7 & 2.74 & 1.2e-2 & 1690.60 & 2.4e-2 & 3601.25\\ \cline{1-11}
 pyrim5 & 1135 & 74 & 9.1e-7 & \textbf{2.34}    & 4.5e-7 & 3.40 & 3.4e-2 & 238.43 & 2.4e-3 & 3601.20 \\ \cline{1-11}
 triazines4 & 2666 & 206 & 2.1e-7 & \textbf{10.24}    & 9.8e-7 & 15.49 & 7.8e-2 & 585.70 & 2.8e-2 & 3601.89 \\ \cline{1-11}
  bodyfat7 & 63 & 8 & 3.0e-7 & \textbf{0.72}    & 7.2e-9 & 1.35 & 9.9e-7 & 41.13 & 3.5e-3 & 3612.99 \\ \cline{1-11}
  abalone7 & 1 & 1 & 1.6e-7 & \textbf{0.83}    & 5.3e-8 & 0.95 & 1.3e-3 & 32.51 & 6.4e-4 & 538.90 \\ \cline{1-11}
 housing7 & 205 & 47  & 7.6e-7  &\textbf{1.98}   & 8.2e-7 & 2.73 & 5.0e-3 & 117.07 & 2.2e-2 & 3600.28 \\ \cline{1-11}
  mpg7  & 42 & 20 & 1.9e-7  & \textbf{0.08}   & 1.8e-7 & 0.11 & 3.4e-6 & 3.19 & 6.3e-6 & 156.31 \\ \cline{1-11}
 spacega9  & 24 & 11 & 5.0e-8  & \textbf{0.27}   & 1.2e-7 & 0.44 & 6.1e-8 & 5.32 & 9.9e-7 & 337.14 \\ \cline{1-11}
E2006.train & 1 & 1 & 3.7e-7  & \textbf{0.42}    & 4.0e-8 & 0.98 & 9.7e-12 & 0.39 & 4.3e-5 & 1196.42\\ \cline{1-11}
 \end{tabular}
 \caption{The results of tested algorithms on Fused Lasso problem ($\lambda_1 = 10^{-3}  \|\mathcal{B}^*b\|_{\infty}$ and $\lambda_2 = 5 \lambda_1.$ ) }\label{tab:fus1e3}
\end{table}

\begin{table}[h]
\begin{tabular}{|c|c|c|cc|cc|cc|cc|}
\cline{1-11}
\multirow{2}{*}{id} & \multirow{2}{*}{nnz($x$)}& \multirow{2}{*}{nnz($Bx$)} & \multicolumn{2}{c|}{SSNCVX} & \multicolumn{2}{c|}{SSNAL} & \multicolumn{2}{c|}{SLEP} & \multicolumn{2}{c|}{ADMM} \\
\cline{4-11}
& & & $\eta$ & time & $\eta$ & time & $\eta$ & time & $\eta$ & time\\
\cline{1-11}
uci\_CT & 18 & 8 & 6.3e-7 & \textbf{0.40} & 8.9e-10 & 0.42 & 1.8e-6 & 2.06 & 7.7e-3 & 39.29 \\
\cline{1-11}
log1p.E2006.train & 8 & 3 & 7.0e-7 & \textbf{8.37} & 1.5e-7 & 12.6 & 1.2e-2 & 4889.15 & 1.2e-1 & 3606.14 \\
\cline{1-11}
E2006.test & 1 & 1 & 1.5e-7 & \textbf{0.17} & 2.9e-8 & 0.33 & 4.8e-8 & 0.93 & 7.7e-7 & 699.27 \\
\cline{1-11}
log1p.E2006.test & 32 & 5 & 3.1e-9 & \textbf{3.07} & 1.2e-8 & 3.31 & 1.2e-2 & 1690.60 & 7.9e-2 & 3601.20 \\
\cline{1-11}
pyrim5 & 327 & 97 & 9.1e-7 & \textbf{2.34} & 2.0e-7 & 3.06 & 3.4e-2 & 238.43 & 1.5e-3 & 3601.13 \\
\cline{1-11}
triazines4 & 1244 & 286 & 8.2e-7 & \textbf{10.51} &2.4e-7 & 12.63 & 7.8e-2 & 585.70 & 2.8e-2 & 3603.56 \\
\cline{1-11}
bodyfat7 & 2 & 3 & 2.8e-8 & \textbf{0.81} & 4.7e-8 & 0.89 & 9.9e-7 & 41.13 & 2.7e-3 & 3606.85 \\
\cline{1-11}
abalone7 & 26 & 15 & 3.7e-7 & \textbf{0.49} & 5.0e-9 & 1.17 & 1.3e-3 & 32.51 & 5.0e-4 & 545.23 \\
\cline{1-11}
housing7 & 131 & 117 & 6.4e-7 & \textbf{1.46} & 3.9e-7 & 2.4 & 5.0e-3 & 117.07 & 2.0e-2 & 3603.08 \\
\cline{1-11}
mpg7 & 32 & 39 & 6.7e-7 & \textbf{0.07} & 2.2e-7 & 0.15 & 3.4e-6 & 3.19 & 1.0e-6 & 77.58 \\
\cline{1-11}
spacega9 & 14 & 13 & 8.7e-7 & \textbf{0.22} & 1.7e-7 & 0.44 & 6.1e-8 & 5.32 & 1.0e-6 & 333.39 \\
\cline{1-11}
E2006.train & 1 & 1 & 4.2e-7 & \textbf{0.45} & 4.0e-7 & 1.12 & 9.7e-12 & 0.39 & 4.4e-5 & 1189.36 \\
\cline{1-11}
\end{tabular}
\caption{The results of tested algorithms on Fused Lasso problem ($\lambda_1 = 10^{-3} \|\mathcal{B}^*b\|_{\infty}$ and $\lambda_2 = \lambda_1.$ ) }\label{tab:fus1e31}
\end{table}

\subsection{QP}
The QP problem is also a special case of \eqref{general}. In this subsection, we consider solving portfolio optimization, an application of QP and is widely used in the investment community:
\begin{equation} \label{prob-QP}
    \min_{\bm{x}}  \iprod{\bm{x}}{\mathcal{Q}(\bm{x})} + \iprod{\bm{c}}{\bm{x}}, \quad \st ~~ \iprod{\bm{e}_n}{\bm{x}} = 1, ~~ \bm{x} \ge \bm{0},
\end{equation}
where $\bm{x}$ denotes the decision variable, $\mathcal{Q} \in \mathcal{S}_+^n$ denotes the data matrix, $\gamma > 0$, and $\bm{e}_n$ is the vector of ones. The $\mathcal{Q}$ and $\bm{c}$ are chosen from Maros-M\'{e}sz\'{a}ros dataset \cite{qpbenchmark} and synthetic data. For Maros-M\'{e}sz\'{a}ros dataset, we choose the problem whose dimension is more than 10000 since the data is highly sparse.
For synthetic data, we generate our test data randomly via the following Matlab script as follows \cite{liang2022qppal}:
\begin{center}
\begin{lstlisting}[language=Matlab]
p = 0.01*n;
F = sprandn(n, p, 0.1); D = sparse(diag(sqrt(p)*rand(n,1)));
Q = cov(F') + D;
c = randn(n,1);
\end{lstlisting}
\end{center}
where $\texttt{n}$ denots the dimision. We compare SSNCVX with the HIGHS~\cite{huangfu2018parallelizing} solver. The results are listed in Table \ref{tab:qp}. It is shown that SSNCVX can solve all the tested problems while HiGHS can not.

\begin{table}[!htb]
    \setlength{\tabcolsep}{4pt}
    \centering
\begin{tabular}{|c|c|c|c|c|c|c|c|c|}
    \hline
    &\multicolumn{3}{c|}{SSNCVX}&\multicolumn{3}{c|}{HIGHS} \\ \hline
    problem &  obj & $\eta$  & time &obj&$\eta$&time \\ \hline
    Aug2D & -1.0e+0 & 2.9e-11& \textbf{0.25}&-&-&-\\\hline
    Aug2DC & -1.0e+0 & 7.5e-13 & \textbf{0.20}&-&-&-\\\hline
    Aug2DCQP & -1.0e+0 & 7.5e-13 & \textbf{0.18}&-&-&-\\\hline
    Aug2DQP & -1.0e+0 & 1.7e-16 & \textbf{0.31}&-&-&-\\\hline
    BOYD1& -1.1e+4 & 2.0e-7 & \textbf{47.80}&-&-&-\\\hline
    BOYD2 & -1.0e+1 & 2.3e-9 & \textbf{0.29}&-1.0e+1&4.3e-6 &3667.91\\\hline
    CONT-100 & -3.3e-4 & 7.0e-12&  \textbf{1.23}&-3.3e-4&7.8e-4 &122.04\\\hline
    CONT-101 & -9.9e-5 &  0.0e+0 & \textbf{0.07}&-9.9e-5&4.5e-3 &3600.04\\\hline
    CONT-200 & -8.3e-5 &  4.3e-8 & \textbf{3.96}&-8.3e-5&3.2e-3 &3600.09\\\hline
    CONT-201 & -2.5e-5 &  0.0e+0 & \textbf{0.16}&-&-&-\\\hline
    CONT-300 & -1.1e-5 &  0.0e+0 & \textbf{0.24}&-1.1e-5&0.0e+0 &4011.90\\\hline
    DTOC-3 & 1.3e-8 & 8.9e-18 & \textbf{0.39}&-&-&-\\\hline
    LISWET1 & -1.1e+0 &  2.3e-18 & \textbf{0.15}&-1.1e+0&6.8e-6 &0.70\\\hline
    UBH1 & -0.0e+0 & 4.8e-9&  \textbf{0.28}&-&-&-\\\hline
    random512\_1 & -2.6e+0 &  7.2e-11 & \textbf{0.36}&-2.6e+0&2.1e-7 &1.10\\\hline
    random512\_2 & -2.2e+0 & 7.4e-13 & \textbf{0.40}&-2.2e+0&2.5e-7 &1.11\\\hline
    random1024\_1 & -2.3e+0 & 2.2e-9 & \textbf{1.41}&-2.3e+0&4.0e-7 &2.32\\\hline
    random1024\_2 & -2.5e+0 & 2.7e-8 & \textbf{0.81}&-2.5e+0&2.5e-7 &2.32\\\hline
    random2048\_1 & -2.6e+0 & 1.7e-7 & \textbf{3.40}&-2.6e+0&2.5e-7 &3.96\\\hline
    random2048\_2 & -2.2e+0 & 2.6e-10 & \textbf{2.92}&-2.2e+0&1.4e-7 &4.06\\\hline
\end{tabular}
    \caption{Computational results of tested algorihtms on portfolio optimization.}\label{tab:qp}
\end{table}

\subsection{SOCP}

The SOCP problem corresponding to \eqref{general} is formulated as:
\begin{equation} \label{socp2}
\begin{aligned}
    \min_{\bm{x}} \iprod{\bm{c}}{\bm{x}} \quad \st \, \mathcal{A}(\bm{x}) = \bm{b},~~ \bm{x} \in \mathcal{Q}^n,
\end{aligned}
\end{equation}
where $\mathcal{Q}^n = \mathcal{Q}_1 \times \mathcal{Q}_2 \times \cdots \times \mathcal{Q}_n$ and $\mathcal{Q}_i = \{(x_0, \bar{x}) \in \mathbb{R}^{n_i} | x_0 \geq \|\bar{x}\|_2 )\}$ represents second-order cone.
For the SOCP case, we test the CBLIB problems~\cite{friberg2016cblib} listed in Hans Mittelmann's SOCP Benchmark~\cite{mittelmann2003independent}. Table~\ref{tab:socp_results} compares the running time of SSNCVX with the commonly used solvers ECOS~\cite{domahidi2013ecos}, SDPT3~\cite{toh1999sdpt3}, and MOSEK~\cite{mosek} under a $3600$-second time limit.

Note that the MATLAB solvers (SSNCVX and SDPT3) solve the preprocessed datasets with preprocessing time excluded. This preprocessing, which typically requires several seconds, significantly reduced solution times for some instances (e.g., firL2a), making these solvers appear faster for such problems. However, as geometric means are calculated with a $10$-second shift, the exclusion has a negligible impact on the overall results. On these problems, SSNCVX is 70\% faster than SDPT3, though both remain slower than commercial solver MOSEK. Compared with SDPT3, SSNCVX also exhibits the additional advantage of handling sparse and dense columns separately. Notably, SSNCVX can solve problems like beam7 if we don't set a time limit, while SDPT3 fails due to out of memory.

\begin{scriptsize}
\begin{table}[h]

\begin{tabular}{|l|cc|cc|cc|cc|}
 \hline
 \multirow{2}{*}{id} & \multicolumn{2}{c|}{SSNCVX} & \multicolumn{2}{c|}{SDPT3 } & \multicolumn{2}{c|}{ECOS} & \multicolumn{2}{c|}{MOSEK} \\ \cline{2-9}
  & $\eta$ & time & $\eta$ & time & $\eta$ & time & $\eta$ & time \\ \hline
 beam7 & - & - & - & - & 1.0e-7 & 206.0 & 6.0e-4 & 19.7 \\ \hline
 beam30 & - & - & - & - & 3.0e-7 & 2464.7 & 3.0e-6 & 96.5 \\ \hline
 chainsing-50000-1 & 1.5e-7 & 5.8 & 6.9e-7 & 5.5 & - & - & 1.6e-6 & 3.8 \\ \hline
 chainsing-50000-2 & 7.3e-7 & 14.4 & 7.0e-7 & 9.5 & - & - & 1.0e-7 & 4.1 \\ \hline
 chainsing-50000-3 & 5.0e-9 & 15.7 & 1.4e-7 & 19.4 & - & - & 1.0e-8 & 2.0 \\ \hline
 db-joint-soerensen & - & - & - & - & - & - & 2.0e-8 & 36.3 \\ \hline
 db-plate-yield-line & 8.5e-7 & 597.2 & 8.7e-7 & 217.6 & - & - & 5.0e-7 & 6.2 \\ \hline
 dsNRL & 1.0e-6 & 859.2 & 8.9e-7 & 567.8 & - & - & 8.2e-10 & 67.1 \\ \hline
 firL1 & 5.3e-11 & 101.6 & 7.8e-7 & 582.0 & 3.0e-8 & 1305.2 & 3.1e-9 & 20.5 \\ \hline
 firL1Linfalph & 8.4e-7 & 509.6 & 7.5e-7 & 916.2 & 3.0e-8 & 2846.6 & 4.0e-9 & 91.8 \\ \hline
 firL1Linfeps & 7.0e-7 & 86.4 & 8.2e-7 & 179.1 & 2.0e-9 & 2530.8 & 3.0e-8 & 27.5 \\ \hline
 firL2a & 1.4e-8 & 0.4 & 6.1e-7 & 0.1 & 2.0e-9 & 944.6 & 2.0e-13 & 4.4 \\ \hline
 firL2L1alph & 1.1e-7 & 37.4 & 7.3e-7 & 131.7 & 3.0e-9 & 201.5 & 2.2e-10 & 5.8 \\ \hline
 firL2L1eps & 2.0e-9 & 159.5 & 6.2e-7 & 586.0 & 2.0e-8 & 796.6 & 3.5e-9 & 17.2 \\ \hline
 firL2Linfalph & 7.9e-7 & 89.1 & 7.9e-7 & 799.9 & - & - & 9.0e-9 & 41.7 \\ \hline
 firL2Linfeps & 5.2e-7 & 72.4 & 8.0e-9 & 251.2 & 5.0e-10 & 687.1 & 1.0e-8 & 29.9 \\ \hline
 firLinf & 1.4e-7 & 280.2 & 7.1e-7 & 576.7 & 5.0e-9 & 3478.7 & 1.0e-8 & 123.6 \\ \hline
 wbNRL & 8.7e-7 & 20.1 & 5.9e-7 & 151.2 & 5.0e-9 & 1332.6 & 2.4e-9 & 11.8 \\ \hline
 geomean & - & 155.0 & - & 267.8 & - & 1731.4 & - & 22.7 \\ \hline
\end{tabular}

\caption{The results on Hans Mittelmann's SOCP benchmark.}\label{tab:socp_results}
\end{table}
\end{scriptsize}

\subsection{SPCA}
The sparse PCA problem for a single component is
\[
\max_{\bm{y}} \bm{y}^T \bm{L} \bm{y}, \quad \text{s.t.} \quad \|\bm{y}\|_2 = 1, \quad \text{card}(\bm{y}) \leq k.
\]
The function $\text{card}(\cdot)$ refers to the number of nonzero elements. This problem can be expressed as a low-rank SDP:
\begin{equation}
    \min_{\bm{X}} -\langle \bm{L}  , \bm{X} \rangle + \lambda \|\bm{X} \|_1 ,~  \text{s.t.}~ \text{Tr}(\bm{X}) = 1, \quad \bm{X} \succeq 0.
\end{equation}
 We formulate $\bm{L}$ based on the covariance matrix of real data or use the random example in \cite{zhang2012sparse}. For random examples, $\bm{L}$ is generated by: $ \bm{L} = \frac{1}{\|\bm{u}\|_2} \bm{u} \bm{u}^T + VV^{\mathrm{T}},$
where $\bm{u} = [1, 1/2, \dots, 1/n]$ and each entry of $V \in \mathbb{R}^{n \times n}$ is randomly uniformly chosen from $[0,1]$.
We compare SSNCVX with SuperSCS \cite{sopasakis2019superscs}. The maximum iteration time is set to 3600s. The results are presented in Table \ref{tab:spca}. Compared with SuperSCS, SSNCVX solves SPCA faster and achieves higher accuracy.

\begin{table}[!htb]
    \setlength{\tabcolsep}{4pt}
    \centering
   Here's the modified table with scientific notation using "e" in LaTeX code:
\begin{tabular}{|c|c|c|c|c|c|c|c|}
    \hline
    &\multicolumn{3}{c|}{SSNCVX}&\multicolumn{3}{c|}{superSCS} \\ \hline
    problem &  obj &$\eta$  & time &obj&$\eta_{K}$&time \\ \hline
    20news & -3.3e+3  & 2.0e-12 & 0.8&-3.3e+3&1.0e-6 &9.6\\\hline
    bibtex & -1.8e+4  & 1.2e-11 & 76.6&-1.7e+4&2.7e-1 &3626.4\\\hline
    colon\_cancer & -1.8e+4  & 5.5e-12 & 45.9&-1.4e+4&4.9e-1 &3647.9\\\hline
    delicious & -7.5e+4  & 2.6e-12 & 2.9&-7.5e+4&2.5e-3 &2813.5\\\hline
    dna & -1.8e+3& 1.2e-13 & 0.3&-1.8e+3&1.0e-6 &29.2\\\hline
    gisette & -3.9e+5 & 2.5e-12 & 1190.0&-1.3e+5&7.0e-1 &3703.5\\\hline
    madelon & -9.5e+7  & 5.9e-15 & 16.7&-9.5e+7&4.4e-5 &3343.6\\\hline
    mnist & -2.0e+10 & 4.0e-17 & 15.7&-2.0e+10&1.0e-6 &195.4\\\hline
    protein & -3.0e+3  & 3.5e-11 & 3.7&-3.0e+3&8.7e-3 &2334.1\\\hline
    random1024\_1 & -5.2e+5  & 9.3e-18 & 2.8&-5.3e+5&3.2e-2 &3603.3\\\hline
    random1024\_2 & -5.2e+5  & 4.4e-18 & 2.7&-5.2e+5&1.9e-3 &3604.8\\\hline
    random1024\_3 & -5.2e+5  & 1.3e-17 & 2.8&-5.2e+5&1.4e-3 &3608.3\\\hline
    random2048\_1 & -2.1e+6 & 7.8e-18 & 3.3&-2.0e+6&2.3e-1 &3605.5\\\hline
    random2048\_2 & -2.1e+6 & 5.1e-18 & 3.5&-2.1e+6&5.9e-2 &3607.0\\\hline
    random2048\_3 & -2.1e+6 & 1.5e-18 & 2.3&-2.1e+6&1.5e-2 &3608.2\\\hline
    random4096\_1 & -8.4e+6& 8.2e-18 & 73.4&-1.0e+0&N/A &3655.4\\\hline
    random4096\_2 & -8.4e+6 & 3.5e-18 & 73.1&-8.3e+6&1.2e-2 &3638.0\\\hline
    random4096\_3 & -8.4e+6 & 6.7e-19 & 72.4&-8.4e+6&9.6e-3 &3645.0\\\hline
    random512\_1 & -1.3e+5  & 4.3e-18 & 0.6&-1.3e+5&1.0e-6 &252.0\\\hline
    random512\_2 & -1.3e+5 & 1.1e-17 & 0.6&-1.3e+5&8.1e-3 &2938.5\\\hline
    random512\_3 & -1.3e+5  & 5.7e-18 & 0.6&-1.3e+5&8.2e-3 &2802.0\\\hline
    usps & -1.2e+5 &  2.4e-13 & 1.1&-1.2e+5&1.0e-6 &229.8\\\hline
\end{tabular}
    \caption{Computational results of SSNCVX and superSCS on SPCA.}\label{tab:spca}
\end{table}

\subsection{LRMC}
Low-rank matrix recovery (LRMC) is a classical problem in image processing \cite{wen2012solving}.
The LRMC problem corresponding to \eqref{general} is represented by
\begin{equation} \label{prob:LRMC}
\min_{\bm{X}} \|\mathcal{B}(\bm{X}) - \bm{B}\|_{\mathrm{F}}^2 + \lambda \|\bm{X}\|_*.
\end{equation}
We compare SSNCVX with classical ADMM, proximal gradient, and accelerated proximal gradient method on 8 images. The tested images are listed in Figure \ref{fig:three_subfigures2}. The tested images are corrupted by randomly choosing 50 percent of the pixels. The results are listed in Table \ref{tab:LRMC}. It is shown that SSNCVX not only has higher accuracy but also is the fastest compared with the tested first-order methods.


\begin{figure}[H]
    \centering
    {\includegraphics[width=0.113
    \textwidth]{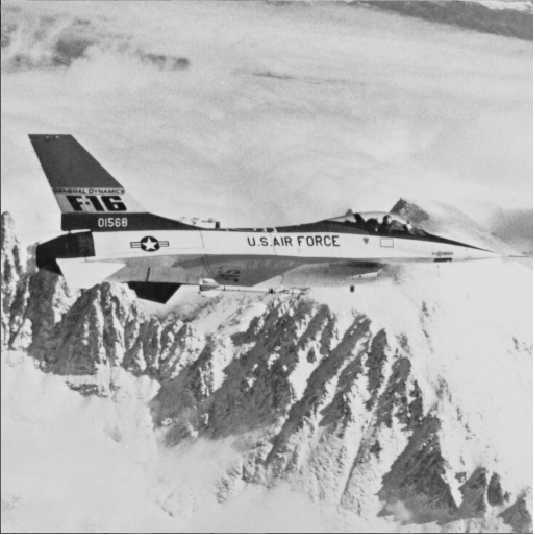} }
    {\includegraphics[width=0.113\textwidth]{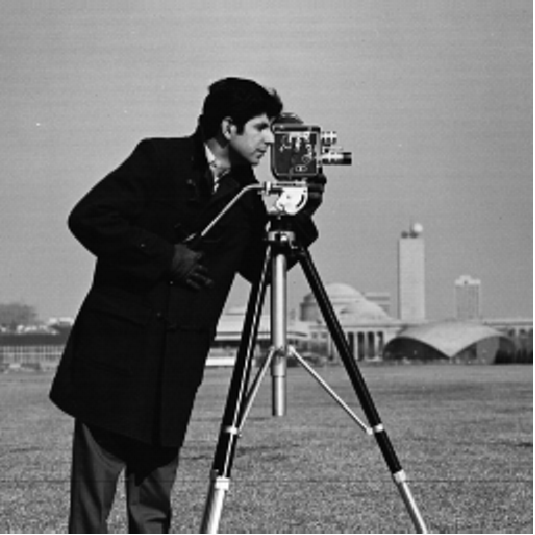} }
    {\includegraphics[width=0.113\textwidth]{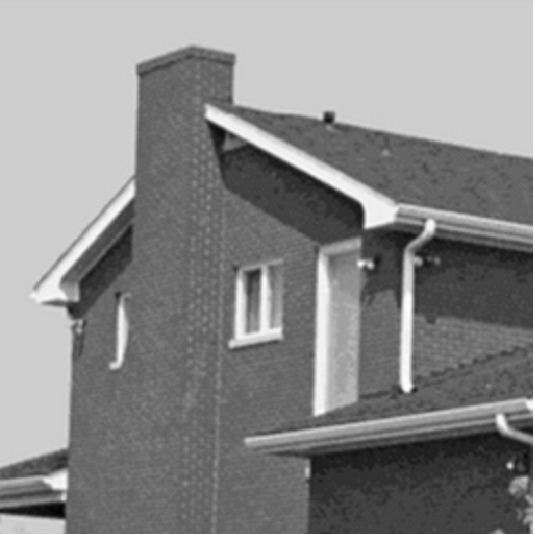} }
    {\includegraphics[width=0.113\textwidth]{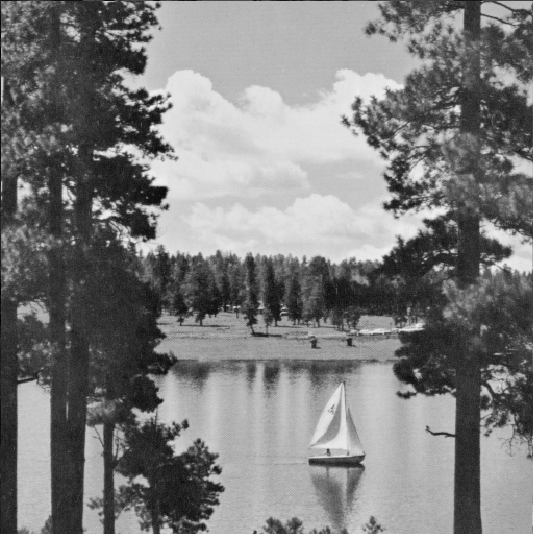}}
    {\includegraphics[width=0.113\textwidth]{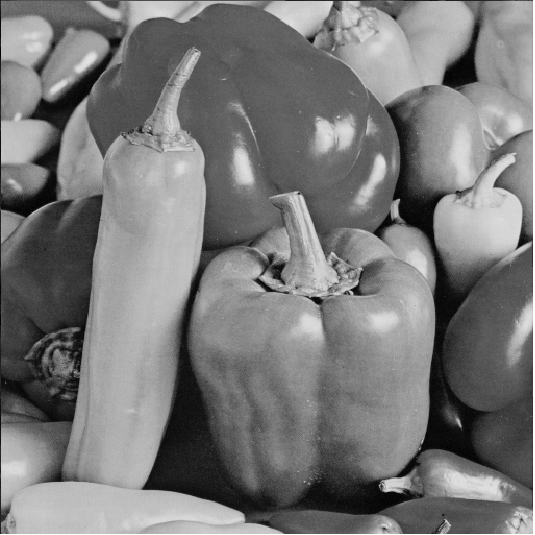} }
    {\includegraphics[width=0.113\textwidth]{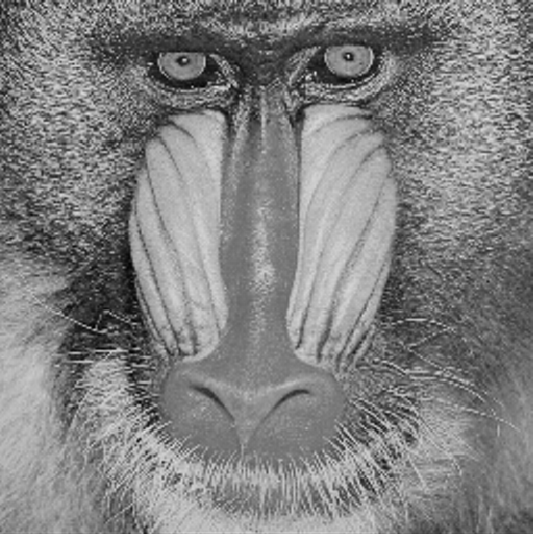} }
    {\includegraphics[width=0.113\textwidth]{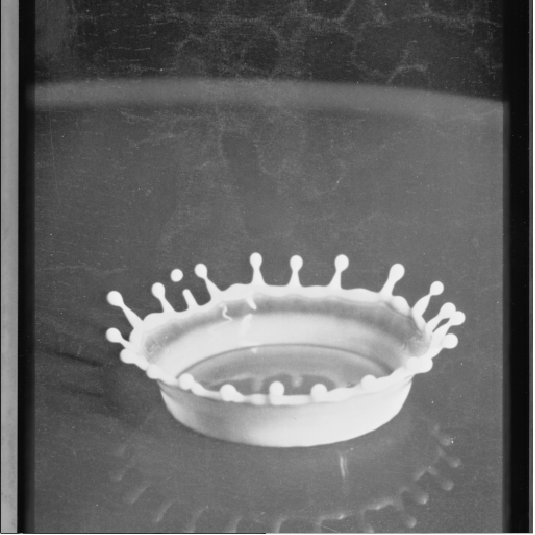} }
    {\includegraphics[width=0.113\textwidth]{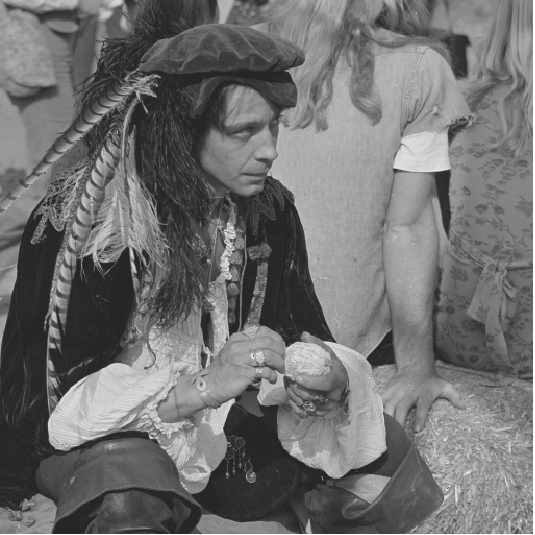}}
    \caption{The eight tested images for LRMC.}
    \label{fig:three_subfigures2}
\end{figure}

\begin{table}[htbp]
\centering
\begin{tabular}{|c|c|c|c|c|c|c|c|c|}
\hline
\multirow{2}{*}{Problem} & \multicolumn{2}{|c|}{SSNCVX} & \multicolumn{2}{|c|}{ADMM} & \multicolumn{2}{|c|}{PG} & \multicolumn{2}{|c|}{APG} \\
\cline{2-9}
 & $\eta$ & Time & $\eta$ & Time & $\eta$ & Time & $\eta$ & Time \\
\hline
Image1 & 1.5e-9 & \textbf{20.1} & 9.9e-9 & 84.5 & 9.6e-9 & 122.4 & 9.9e-9 & 55.8 \\
\hline
Image2 & 4.3e-9 & \textbf{22.1} & 1.0e-8 & 84.0 & 9.8e-9 & 120.9 & 9.6e-9 & 54.5 \\
\hline
Image3 & 5.3e-9 & \textbf{23.2} & 9.9e-9 & 82.8 & 9.6e-9 & 119.5 & 9.3e-9 & 53.9 \\
\hline
Image4 & 3.3e-9 & \textbf{25.3} & 9.7e-9 & 84.1 & 9.8e-9 & 121.1 & 9.8e-9 & 54.6 \\
\hline
Image5 & 7.4e-9 & \textbf{20.3} & 9.5e-9 & 83.7 & 9.7e-9 & 120.4 & 9.9e-9 & 54.4 \\
\hline
Image6 & 1.9e-9 & \textbf{20.9} & 1.0e-8 & 83.5 & 9.8e-9 & 120.4 & 9.7e-9 & 54.3 \\
\hline
Image7 & 1.6e-9 & \textbf{20.2} & 9.9e-9 & 82.2 & 9.9e-9 & 118.3 & 9.7e-9 & 53.1 \\
\hline
Image8 & 2.3e-9 & \textbf{20.8} & 9.8e-9 & 83.0 & 9.7e-9 & 120.0 & 9.7e-9 & 53.9 \\
\hline
\end{tabular}

\caption{ Comparison of tested algorithms on the LRMC problem.}
\label{tab:LRMC}
\end{table}

\section{Conclusion} \label{5}
In this paper, we propose SSNCVX, a semismooth Newton-based algorithmic framework for solving convex composite optimization problems. By reformulating the problem through augmented Lagrangian duality and characterizing the optimality condition via a semismooth equation system, our method provides a unified approach to handle multi-block problems with nonsmooth terms. The framework eliminates the need for problem-specific transformations while enabling flexible model modifications through simple interface updates. Featuring a single-loop structure with second-order semismooth Newton steps, SSNCVX demonstrates superior efficiency and robustness in extensive numerical experiments, outperforming state-of-the-art solvers across various applications. Numerical experiments on various problems establish SSNCVX as an effective and versatile tool for large-scale convex optimization.

\bibliographystyle{siamplain}
\bibliography{ref}

\end{document}